\documentclass[12pt]{amsart}
\usepackage{amscd,amssymb,latexsym}
\usepackage{fullpage}
\newcommand{\dimQ}{\mathrm{dim}_{\Q}}

\newcommand{\growth}{\mathrm{growth}}
\newcommand{\fg}{\mathcal{FG}}
\newcommand{\chivec}{\mathbf{\chi}}

\newcommand{\cell}{\mathrm{Cell}}
\newcommand{\symm}{\mathrm{Symm}}
\newcommand{\codim}{\mathrm{codim}}
\newcommand{\Der}{\mathbf{D}}   
\newcommand{\chains}{C}
\newcommand{\bbP}{\mathbb{P}}

\newcommand{\Mdef}[2]{\newcommand{#1}{\relax \ifmmode #2 \else $#2$\fi}}

\newcommand{\thick}{\mathrm{thick}}
\newcommand{\loc}{\mathrm{loc}}

\newcommand{\finbuilds}{\models}
\newcommand{\builds}{\vdash}

\newcommand{\tensor}{\otimes}

\newcommand{\lc}{H_{I}}

\newcommand{\sdr}{\rtimes}

\newcommand{\Hom}{\mathrm{Hom}}

\newcommand{\Ext}{\mathrm{Ext}}
\Mdef{\bhom}{\mathbf{\hat{H}om}}

\Mdef{\Mod}{\mathrm{mod}}

\newcommand{\st}{\; | \;}
\newcommand{\hash}{\#}

\newtheorem{thm}{Theorem}[section]
\newtheorem{lemma}[thm]{Lemma}
\newtheorem{prop}[thm]{Proposition}
\newtheorem{cor}[thm]{Corollary}

\theoremstyle{definition}

\newtheorem{defn}[thm]{Definition}

\newtheorem{notation}[thm]{Notation}

\newtheorem{example}[thm]{Example}

\newtheorem{remark}[thm]{Remark}

\newcommand{\qqed}{\qed \\[1ex]}
\renewenvironment{proof}[1][\hspace*{-.8ex}]{\noindent {\bf Proof #1:\;}}{\qqed}

\Mdef{\PH} {\Phi^H}
\Mdef{\PK} {\Phi^K}
\Mdef{\PL} {\Phi^L}
\Mdef{\PT} {\Phi^{\T}}

\Mdef{\ef}{E{\cF}_+}
\Mdef{\etf}{\tilde{E}{\cF}}
\Mdef{\eg}{E{G}_+}
\Mdef{\etg}{\tilde{E}{G}}

\Mdef{\infl}{\mathrm{inf}}
\Mdef{\defl}{\mathrm{def}}
\Mdef{\res}{\mathrm{res}}
\Mdef{\ind}{\mathrm{ind}}

\Mdef{\univ}{\mathcal{U}}

\newcommand{\CP}{\mathbb{C} P}

\Mdef{\Fp}{\mathbb{F}_p}
\Mdef{\Zpinfty}{\Z /p^{\infty}}
\Mdef{\Zpadic}{\Z_p^{\wedge}}

\newcommand{\bi}{\begin{itemize}}
\newcommand{\be}{\begin{enumerate}}
\newcommand{\bc}{\begin{center}}
\newcommand{\bd}{\begin{description}}
\newcommand{\ei}{\end{itemize}}
\newcommand{\ee}{\end{enumerate}}
\newcommand{\ec}{\end{center}}
\newcommand{\ed}{\end{description}}

\newcommand{\adjunction}[4]{
\diagram
#1:#2 \rrto<0.7ex> &&
#3  \llto<0.7ex> :#4
\enddiagram}
\newcommand{\lra}{\longrightarrow}
\newcommand{\lla}{\longleftarrow}

\Mdef{\we}{\mathbf{we}}
\Mdef{\fib}{\mathbf{fib}}
\Mdef{\cof}{\mathbf{cof}}
\Mdef{\BI}{\mathcal{BI}}

\newcommand{\fibre}{\mathrm{fibre}}

\newcommand{\ilim}{\mathop{ \mathop{\mathrm{lim}} \limits_\leftarrow} \nolimits}

\Mdef{\A}{\mathbb{A}}
\Mdef{\B}{\mathbb{B}}
\Mdef{\C}{\mathbb{C}}
\Mdef{\D}{\mathbb{D}}
\Mdef{\E}{\mathbb{E}}
\Mdef{\T}{\mathbb{T}}
\Mdef{\F}{\mathbb{F}}
\Mdef{\G}{\mathbb{G}}
\Mdef{\I}{\mathbb{I}}
\Mdef{\N}{\mathbb{N}}
\Mdef{\Q}{\mathbb{Q}}
\Mdef{\R}{\mathbb{R}}
\Mdef{\bbS}{\mathbb{S}}
\Mdef{\Z}{\mathbb{Z}}

\Mdef{\bA}{\mathbb{A}}
\Mdef{\bB}{\mathbb{B}}
\Mdef{\bC}{\mathbb{C}}
\Mdef{\bD}{\mathbb{D}}
\Mdef{\bE}{\mathbb{E}}
\Mdef{\bF}{\mathbb{F}}
\Mdef{\bG}{\mathbb{G}}
\Mdef{\bH}{\mathbb{H}}
\Mdef{\bI}{\mathbb{I}}
\Mdef{\bJ}{\mathbb{J}}
\Mdef{\bK}{\mathbb{K}}
\Mdef{\bL}{\mathbb{L}}
\Mdef{\bM}{\mathbb{M}}
\Mdef{\bN}{\mathbb{N}}
\Mdef{\bO}{\mathbb{O}}
\Mdef{\bP}{\mathbb{P}}
\Mdef{\bQ}{\mathbb{Q}}
\Mdef{\bR}{\mathbb{R}}
\Mdef{\bS}{\mathbb{S}}
\Mdef{\bT}{\mathbb{T}}
\Mdef{\bU}{\mathbb{U}}
\Mdef{\bV}{\mathbb{V}}
\Mdef{\bW}{\mathbb{W}}
\Mdef{\bX}{\mathbb{X}}
\Mdef{\bY}{\mathbb{Y}}
\Mdef{\bZ}{\mathbb{Z}}

\Mdef{\cA}{\mathcal{A}}
\Mdef{\cB}{\mathcal{B}}
\Mdef{\cC}{\mathcal{C}}
\Mdef{\mcD}{\mathcal{D}} 
\Mdef{\cE}{\mathcal{E}}
\Mdef{\cF}{\mathcal{F}}
\Mdef{\cG}{\mathcal{G}}
\Mdef{\mcH}{\mathcal{H}} 
\Mdef{\cI}{\mathcal{I}}
\Mdef{\cJ}{\mathcal{J}}
\Mdef{\cK}{\mathcal{K}}
\Mdef{\mcL}{\mathcal{L}}

\Mdef{\cM}{\mathcal{M}}
\Mdef{\cN}{\mathcal{N}}
\Mdef{\cO}{\mathcal{O}}
\Mdef{\cP}{\mathcal{P}}
\Mdef{\cQ}{\mathcal{Q}}
\Mdef{\mcR}{\mathcal{R}}
\Mdef{\cS}{\mathcal{S}}
\Mdef{\cT}{\mathcal{T}}
\Mdef{\cU}{\mathcal{U}}
\Mdef{\cV}{\mathcal{V}}
\Mdef{\cW}{\mathcal{W}}
\Mdef{\cX}{\mathcal{X}}
\Mdef{\cY}{\mathcal{Y}}
\Mdef{\cZ}{\mathcal{Z}}

\Mdef{\tA}{\tilde{A}}
\Mdef{\tB}{\tilde{B}}
\Mdef{\tC}{\tilde{C}}
\Mdef{\tE}{\tilde{E}}
\Mdef{\tH}{\tilde{H}}
\Mdef{\tK}{\tilde{K}}
\Mdef{\tL}{\tilde{L}}
\Mdef{\tM}{\tilde{M}}
\Mdef{\tN}{\tilde{N}}
\Mdef{\tP}{\tilde{P}}

\Mdef{\tf}{\tilde{f}}

\Mdef{\Ab}{\overline{A}}
\Mdef{\Bb}{\overline{B}}
\Mdef{\Cb}{\overline{C}}
\Mdef{\Db}{\overline{D}}
\Mdef{\Eb}{\overline{E}}
\Mdef{\Hb}{\overline{H}}
\Mdef{\Gb}{\overline{G}}
\Mdef{\Ib}{\overline{I}}
\Mdef{\Kb}{\overline{K}}
\Mdef{\Lb}{\overline{L}}
\Mdef{\Mb}{\overline{M}}
\Mdef{\Nb}{\overline{N}}
\Mdef{\Qb}{\overline{Q}}
\Mdef{\Tb}{\overline{T}}

\Mdef{\db}{\overline{d}}
\Mdef{\hb}{\overline{h}}
\Mdef{\qb}{\overline{q}}
\Mdef{\rb}{\overline{r}}
\Mdef{\tb}{\overline{t}}
\Mdef{\ub}{\overline{u}}
\Mdef{\vb}{\overline{v}}

\Mdef{\hc}{\hat{c}}
\Mdef{\he}{\hat{e}}
\Mdef{\hf}{\hat{f}}
\Mdef{\hA}{\hat{A}}
\Mdef{\hH}{\hat{H}}
\Mdef{\hJ}{\hat{J}}
\Mdef{\hM}{\hat{M}}
\Mdef{\hP}{\hat{P}}
\Mdef{\hQ}{\hat{Q}}

\Mdef{\thetab}{\overline{\theta}}
\Mdef{\phib}{\overline{\phi}}

\Mdef{\uA}{\underline{A}}
\Mdef{\uB}{\underline{B}}
\Mdef{\uC}{\underline{C}}
\Mdef{\uD}{\underline{D}}

\Mdef{\bolda}{\mathbf{a}}
\Mdef{\boldb}{\mathbf{b}}
\Mdef{\boldD}{\mathbf{D}}

\Mdef{\fm}{\frak{m}}

\Mdef{\eps}{\epsilon}

\usepackage{hyperref}
\input{xypic}

\begin{document}
\title{Complete intersections in rational homotopy theory.}

\author{J.P.C.Greenlees}
\address{School of Mathematics and Statistics, Hicks Building,
Sheffield S3 7RH. UK.}
\email{j.greenlees@sheffield.ac.uk}
\date{}

\author{K.Hess}
\address{EPFL, Lausanne, Switzerland}
\email{kathryn.hess@epfl.ch}

\author{S.Shamir}
\address{School of Mathematics and Statistics, Hicks Building,
Sheffield S3 7RH. UK.}
\email{s.shamir@sheffield.ac.uk}
\date{}

\begin{abstract}
We investigate various homotopy invariant formulations of commutative
algebra in the context of rational homotopy theory. The main subject
is the complete intersection condition, where we show that a growth
condition implies a structure theorem and that modules have multiply
periodic resolutions.
\end{abstract}

\thanks{This work was supported by EPSRC Grant number EP/E012957/1}
\maketitle

\tableofcontents

\section{Introduction}
\label{sec:Intro}

\subsection{Background.}
It has been very fruitful to adapt the definitions of commutative algebra
so that they apply in homotopy theory. The original motivation is
that it is useful to study a space $X$ through a ring of functions,
and for our purposes we will think of the ring $C^*(X;k)$ of cochains on $X$.
Of course, if the analogy is to be accurate, we need a commutative model
for $C^*(X;k)$, and if it is to be effective we need to render the definitions
homotopy invariant.

The prime example of this  is the connection between rational homotopy
theory and rational differential graded algebras (DGAs), but the availability of
good models for ring spectra has led to other useful examples in positive
characteristic. The emphasis in classical rational homotopy theory has
been on finite complexes and calculation, whereas one of the themes in
characteristic $p$ has been to consider classifying spaces of compact Lie groups
where the natural finiteness condition is that the cohomology rings in
question are Noetherian.
The purpose of the present paper is to take the ideas developed for
compact Lie groups and investigate them in the much more accessible
context of rational homotopy theory. From one point of view this is the
process of generalizing classical results \cite{FHT} from the case when
$H^*(X;\Q)$ is finite dimensional to the case when it is Noetherian, and given the available tools of rational
homotopy theory, this is reasonably straightforward. From another point
of view this is an opportunity to give new and accessible examples of
the theory, and to test expectations in a context where complete calculation
is often possible. Finally, the work suggests a number of questions we may
translate through the mirror \cite{AH} to local algebra, and we plan to investigate
these in future.
\subsection{Contents.}
On the commutative algebra side we restrict attention to commutative, local,
Noetherian rings. On the topological side, we restrict attention to
simply connected, rational spaces $X$ with $H^*(X)$ Noetherian.

We begin by considering analogues of regular and Gorenstein local rings.
Both are already well-known in rational homotopy theory, but it gives us an
opportunity to introduce some terminology and to express things in a
convenient language. For example,  we emphasize the importance of homotopy
invariant finiteness conditions and
 Morita theory. The regular spaces $X$ are precisely those which are
finite products of even Eilenberg-MacLane spaces. There are enormous
numbers of Gorenstein spaces, and they include manifolds and  all finite
Postnikov systems.

In the classical literature of rational homotopy theory, Gorenstein
 {\em duality} does not seem to be a familiar phenomenon  except in the
zero-dimensional cases (Poincar\'e Duality). We take the opportunity
in Appendix A to explain how the Local Cohomology Theorem from
\cite{DGI1} gives a Gorenstein duality statement in general.
From the point of view of rational
homotopy theory it shows (for example) that if $X$ is
any finite Postnikov system and $H^*(X)$ is Cohen-Macaulay, it is
automatically Gorenstein. Furthermore, without any hypothesis on the depth,
 $H^*(X)$ is generically Gorenstein. From the point of view of
homotopy invariant commutative algebra, it gives an extremely rich and
flexible source of examples.

The main subject of the paper is a study of the complete intersection
(ci) condition. We give a number of homotopy invariant definitions of ci
spaces, corresponding to different aspects of the ci condition. These
have very different characters, so it is striking that we are able to
show that in the rational context they are all equivalent.
The {\em structural} condition in commutative algebra is
that a  ci ring is a quotient of a regular local ring by a regular sequence.
We say that a simply connected rational space $X$ is  sci if it
is formed from a finite product of even Eilenberg-MacLane spaces by iterated
spherical fibrations (all definitions are given precisely in
Section \ref{sec:ci}).
Secondly, Benson and the first author \cite{BGzci}
introduced a finiteness condition (zci) on the category of modules analogous
to requiring all modules to have eventually multiply periodic resolutions.
This is a strengthening of
the condition in \cite{DGI2}.
 {In this paper we needed to relax the zci condition to two new finiteness
conditions, the eci and the nci conditions. }
Finally, in commutative algebra
there is the {\em growth} condition that $\Ext_R^*(k,k)$ has polynomial
growth (equivalent to the structural condition by Gulliksen's theorem); the
condition on a rational simply connected space is the growth condition
(gci) that $H_*(\Omega X)$ has polynomial growth.

Most  remarkable of the equivalences, perhaps, is the fact that the
growth condition implies a structure theorem:  $X$ is gci if and
only if  there is a fibration
$$F\lra X \lra KV,$$
where $KV$ is a finite product of even Eilenberg-MacLane spaces and
$\pi_*(F)$ is entirely in odd degrees.  {Amongst these spaces, those
in which $F$ has trivial $k$-invariants, so that $F$ is a product of odd spheres,
are the ones with pure Sullivan models.}

Another unexpected phenomenon is the importance of the Noetherian
condition. On the one hand, an iterated spherical fibration over a
product of even Eilenberg-MacLane spaces is obviously Noetherian.
One might naively think that requiring $H_*(\Omega X)$ to have polynomial growth
would be enough without requiring $H^*(X)$ to be Noetherian, but in
fact the Milnor-Moore theorem shows that this
just means $\pi_*(X)$ is finite dimensional. It is very striking that
the Noetherian condition is sufficient to give a structure theorem, and
we are grateful to  N.P.Strickland for a timely remark. We also thank
S.B.Iyengar for comments.

\subsection{The layout of the paper. }
After summarizing conventions in Section \ref{sec:conventions},
we begin in Section \ref{sec:Qhtpy} by giving a brief summary of the results
and terminology we need from rational homotopy theory.
Next, in Section \ref{sec:Morita} we describe the Morita theory for moving
between $C^*(X)$ and $C_*(\Omega X)$, and some results on cellularization
from \cite{DGI1}. We are then in a position to
consider rational DGAs in parallel with rational spaces. In a series of
sections we describe the definitions for rational DGAs and in particular for
Sullivan models of rational spaces. In Section \ref{sec:regular} we consider
regular rings and spaces, and in Appendix \ref{sec:Gorenstein} we discuss
Gorenstein spaces and Gorenstein duality.

From Section \ref{sec:centre} onwards, our main concern is for complete
intersections. First, Section \ref{sec:centre} discusses the centre of a
derived category, and how bimodules and Hochschild cohomology give elements
of the centre. Section \ref{sec:ci} introduces the definitions designed
to capture various aspects of hypersurface and ci spaces, which later
sections show to be equivalent. Section \ref{sec:sci} takes the structural
definition, and shows that any sci space has a standard form. Section
\ref{sec:zciisgci} gives the elementary argument that zci spaces satisfy
the gci growth condition. 
 {Section \ref{sec:sciisbci}
shows that sci spaces all have eventually multiply periodic
module theories. In Section \ref{sec:bciiszci} we calculate the Hochschild cohomology
of all pure sci spaces relative to their regular base and use the result to
show they are zci. }
Finally, and perhaps most
interestingly, in Section \ref{sec:gciissci} we show that the growth condition alone is enough to show that a space
has the standard sci form. Section \ref{sec:examples} gives a number of
explicit examples illustrating the phenomena we have studied, and showing
that the various classes of spaces are distinct. The final section explores
the role of the Noetherian condition further, and gives a characterisation
of the polynomial growth of $H_*(\Omega X)$ when we do not require
 $H^*(X)$ to be Noetherian in the same style as the multiply periodic
resolution property for ci spaces.

\section{Conventions.}
\label{sec:conventions}
\subsection{Terminology for triangulated categories.}
Recall that an object $X$ of a triangulated category $\cT$ is called
{\em small} if the natural map
$$\bigoplus_i [X,Y_i] \lra [X,\bigvee_i Y_i]$$
is an isomorphism for any set of objects $Y_i$.

A {\em thick} subcategory of
$\cT$ is a full subcategory closed under completion of triangles and taking
retracts. We write $\thick (X)$ for the smallest thick subcategory containing
$X$, and if $A \in \thick (X)$ we also say
`{\em $X$ finitely builds  $A$}' and write $X\finbuilds A$.

A {\em localizing } subcategory of
$\cT$ is a thick subcategory which is also closed under taking
arbitrary coproducts. We write $\loc (X)$ for the smallest localizing
 subcategory containing
$X$, and if $A\in \loc (X)$ we also say `{\em $X$ builds $A$}'
and write $X\builds A$.

Following \cite{DGI2} we say that $X$ is {\em virtually small} if
$\thick (X)$ contains a non-trivial small object $W$, and we say that
any such $W$ is a {\em witness} for the fact that $X$ is virtually small.

\subsection{Grading conventions.}
We will have cause to discuss homological and cohomological gradings.
Our experience is that this a frequent source of confusion, so we adopt
the following conventions. First, we refer to lower gradings as {\em degrees}
and upper gradings as {\em codegrees}. As usual, one may convert gradings to
cogradings via the rule $M_n=M^{-n}$. Thus both chain complexes and cochain
complexes have differentials of degree $-1$ (which is to say, of codegree 1).
This much is standard. However, since we need to deal with both chain complexes
and cochain complexes it is essential to have separate notation for
homological suspensions ($\Sigma^i$) and  cohomological suspensions
($\Sigma_i$): these are defined by
$$(\Sigma^iM)_n=M_{n-i} \mbox{ and } (\Sigma_iM)^n=M^{n-i}.$$
Thus, for example, with reduced chains and cochains  of a based
space $X$,  we have
$$\tilde{C}_*(\Sigma^iX)=\Sigma^i\tilde{C}_*(X) \mbox{ and }
\tilde{C}^*(\Sigma^iX)=\Sigma_i\tilde{C}^*(X).$$

\subsection{Other conventions.}
Unless explicitly stated to the contrary, all coefficients will
be in the rational numbers $\Q$, and for a rational vector space
$V$, we write $V^{\vee}=\Hom_{\Q}(V, \Q)$ for the dual vector space.

For brevity we write CGA for {\em commutative graded algebra}
(i.e., an algebra which is commutative in the graded sense
that $xy=(-1)^{|x|\cdot |y|}yx$),  DGA for {\em differential graded algebra},
and  CDGA for {\em commutative differential graded algebra}. When
we refer to modules over a DGA, we intend differential graded modules
unless otherwise stated.

Finally, for a space $X$, we write $C^*(X)$ for a CDGA model for
the cochains on $X$.

\section{Rational homotopy theory.}
\label{sec:Qhtpy}
Rational homotopy theory provides the ideal context to test
ideas about homotopy invariant commutative algebra. On the one hand
many aspects of commutative algebra are especially simple for
$\Q$-algebras and  on the other we can appeal to the intuition and
structures of homotopy theory.

\subsection{Terminology for commutative differential graded algebras.}
We will restrict attention to simply connected $\Q$-algebras of
finite type.

If $V$ is a graded rational vector space, we write $\Lambda (V)$
for the free CGA on $V$. This is a symmetric algebra on $V^{ev}$
tensored with an exterior algebra on $V^{od}$. A {\em Sullivan algebra}
is a CDGA which is free as a CGA on a simply connected graded vector
space $V$ of finite dimension in each degree, and whose differential
has the property that if $x \in V^s$ then $dx \in \Lambda (V^{<s})$.
It is minimal if in addition $d$ takes values in $\Lambda^+(V)^2$.
 {
A Sullivan algebra $(\Lambda(V^{od}\oplus V^{ev}),d)$ is
{\em pure} if $d(V^{od})\subset \Lambda V^{ev}$ and $d(V^{ev})=0$.
}

A {\em relative Sullivan algebra} is a map $M \lra M\sdr \Lambda (V)$
of CDGAs. Here the underlying CGA of $M\sdr \Lambda (V)$ is
$M\tensor \Lambda (V)$, and the notation expresses the fact that $M$ is
a sub-DGA and $\Lambda (V)$ is a quotient.

\subsection{Rational models for simply connected spaces.}
Any simply connected rational CW-complex with cohomology finite in each
degree is modelled by a simply connected rational CDGA (such as
the CDGA of PL polynomial differential forms $\cA_{PL} (X)$). Furthermore, any such CDGA has a Sullivan minimal
model, unique up to isomorphism. We write $C^*(X)$ for an unspecified
CDGA model for the cochains on $X$.
The process of building up a Sullivan algebra degree by degree
corresponds to building up a space using a Postnikov tower.

If $V$ is an evenly graded vector space, we write $KV$ for the associated
Eilenberg-MacLane space. In principle we could use the same notation
when $V$ has an odd summand, but we will not do so. Since odd spheres
are rational Eilenberg-MacLane spaces, if $W$ is a graded vector space in
odd degrees, we write $S(W)$ for the corresponding Eilenberg-MacLane space.

 {
We say a space $X$ is {\em pure} if $X$ has pure Sullivan algebra model.
}

A fibration $E \lra B$ with fibre $F$ can be modelled by a relative Sullivan
algebra $M\sdr \Lambda (V) \lla M$ where $M$ models $B$,
$M\sdr \Lambda (V)$ models $E$ and the fibre $F$ is then modelled by
$\Lambda (V)$.

\subsection{Homotopy Lie algebras and the Milnor-Moore theorem.}

Recall that $\pi_*(\Omega X)$ is a graded Lie algebra under the
Samelson product. More precisely there is a natural bilinear product
$$[\cdot, \cdot ]: \pi_i(\Omega X)\times \pi_j(\Omega X)\lra \pi_{i+j}
(\Omega X)$$
which is antisymmetric in the sense that
$$[x,y]=-(-1)^{|x|\cdot |y|}[y,x]$$
and satisfies the graded Jacobi identity
$$(-1)^{|x|\cdot |z|}[x,[y,z]]+(-1)^{|y|\cdot |x|}[y,[z,x]]+(-1)^{|z|\cdot |y|}[z,[x,y]]=0. $$
One way of forming a graded Lie algebra from an associative algebra $A$
is to define $[x,y]=x y -(-1)^{|x|\cdot |y|}y x$ for homogeneous elements
$x,y \in A$.
Associated to a graded Lie algebra is a universal associative algebra
$$U(L)=TL/I$$
where $TL$ is the tensor algebra on $L$ and $I$ is the ideal generated
by the relations
$[x,y]=x\tensor y -(-1)^{|x|\cdot |y|}y\tensor x$ for $x,y \in L$.
Henceforth we will generally omit the notation for the tensor product.

The most important algebraic fact about the universal enveloping algebra
of a Lie
algebra is the Poincar\'e-Birkoff-Witt theorem stating that if we
filter $U(L)$ by tensor length then there is an isomorphism
$$Gr (U(L))=\Lambda L. $$
In particular, the  growth rate of $U(L)$ is the same
as that of the symmetric algebra on $L^{ev}$.

The following theorem makes this relevant to topology.

\begin{thm} (Milnor-Moore \cite{MM})
If $X$ is a simply connected rational space then
$$H_*(\Omega X)=U(\pi_*(\Omega X)).\qqed$$
\end{thm}

In particular, we see that $H_*(\Omega X)$ has polynomial
growth if and only if $\pi_*(\Omega X)$ is finite dimensional, and
in that case the growth is of degree one less than
$$\dimQ (\pi_{ev}(\Omega X))= \dimQ (\pi_{od}( X)).$$

\subsection{Elliptic spaces.}
Perhaps for historical reasons, classical rational homotopy theory
concentrates on finite complexes, which is to say spaces with
$H^*(X)$ finite dimensional. These correspond to
0-dimensional local rings.

A simply connected rational space $X$ is called {\em elliptic}
if $H^*(X)$ and  $\pi_*(X)$ are both finite dimensional. It is
called {\em hyperbolic} if $\pi_*(X)$ has exponential growth.

A major theorem of rational homotopy theory is the dichotomy theorem
stating that a simply connected rational space with $H^*(X)$ finite
dimensional is either elliptic or hyperbolic. In a sense we will
make precise, elliptic spaces correspond to 0-dimensionsional complete
intersections.


\subsection{Noether normalization.}

Polynomial rings on even degree generators play a special role
in the theory. To start with, they are {\em intrinsically formal}: if
$P$ is a polynomial ring on even degree generators, then if $A$ is
any CDGA with $H^*(A)\cong P$, we have a quasi-isomorphism $A\simeq P$.
Indeed, $P$ has
a useful universal property: for any CDGA
$A$, and any map $\theta:P \lra H^*(A)$ of CGAs, a choice of representative
cycles for the polynomial generators allows us to realize $\theta$
by a map $\tilde{\theta}: P \lra A$ of CDGAs. Not only
are they convenient, we shall see they have a structural role:
polynomial rings on even degree generators provide the class of
CDGAs corresponding to regular local rings.
We think of $KV$ with $V$ even and finite dimensional
as a generalization of the rational classifying space of
a compact connected Lie group.

Polynomial rings can then be used in the study of general
Noetherian rings. Indeed, the Noether normalization theorem
states that if $R$ is a Noetherian connected CGA,  it is
finitely generated as a module over a polynomial subalgebra
$P$ on even degree generators. We will repeatedly use the
following counterpart of this statement.

\begin{prop}
If $X$ is a 1-connected rational space with $H^*(X)$ Noetherian,
there is a fibration
$$F \lra X \lra KV$$
of rational spaces where $V$ is even and  finite dimensional,  and
$H^*(F)$ is finite dimensional.
\end{prop}

\begin{proof}
By Noether normalization, $H^*(X)$ is finite dimensional over
a polynomial algebra $P$ on even degree generators.
Choosing representative cycles, we have a map $P=KV \lra C^*(X)$ of CDGAs
realizing this map in cohomology.  This gives a fibration
$$F \lra X \lra KV.$$
To see  $H^*(F)$ is finite dimensional,  we note that $H^*(X)$ is a
finitely generated $P$-module, and therefore has a finite resolution
by finitely generated free $P$-modules.
\end{proof}

 We refer to this fibration as a Noether normalization
of $X$, and to $F$ as a {\em Noether fibre}
of $X$. The long exact sequence in homotopy shows
that the growth of $\pi_*(X)$ is the same as that of $\pi_*(F)$.

\begin{lemma} {\em (Dichotomy) }
For a space $X$ with $H^*(X)$ Noetherian, either $\pi_*(X)$ is finite
dimensional or it has exponential growth. The homotopy is finite dimensional
if and only if a Noether fibre is elliptic. \qqed
\end{lemma}

This motivates the following extension of the notion of elliptic spaces
to spaces with Noetherian cohomology.

\begin{defn}
A space $X$ is {\em gci} (or satisfies the growth condition for a complete
intersection) if $H^*(X)$ is Noetherian and $\pi_*(X)$ is
finite dimensional.
\end{defn}

These spaces are the principal subject of the present paper, and
we return to them in Section \ref{sec:ci}.

\subsection{Some analogies.}

At the most basic level, cofibre sequences
$$X \lra Y \lra Z$$
of pointed spaces induce (additive) exact sequences
$$C^*(X) \lla C^*(Y) \lla C^*(Z) $$
of reduced cochains. On the other hand, fibrations
$$F \lra E \lra B$$
of spaces induce (multiplicative) exact sequences
$$ C^*(F)\stackrel{EM}\simeq C^*(E) \tensor_{C^*(B)} \Q \lla C^*(E) \lla C^*(B)$$
provided $C^*(B)$ is 1-connected so that an Eilenberg-Moore theorem (EM) holds,
and $C^*(B) \lra C^*(E)$ is a relative Sullivan model so that the tensor
product is derived.

More generally, a homotopy pullback square
$$\begin{array}{ccc}
Z\times_XY&\lra & Z\\
\downarrow&& \downarrow\\
Y&\lra & X\\
\end{array}$$
induces a homotopy pushout square
$$\begin{array}{ccc}
C^*(Z\times_XY)&\lla & C^*(Z)\\
\uparrow&& \uparrow\\
C^*(Y)&\lla & C^*(X)\\
\end{array}$$
in the sense that
$$C^*(Z\times_XY)\simeq C^*(Z)\tensor_{C^*(X)}C^*(Y)$$
if $X$ is $1$-connected, and one of the maps $C^*(X) \lra C^*(Z)$
or $C^*(X) \lra C^*(Y)$ is a relative Sullivan algebra so that the
tensor product is derived.


We should also record the Rothenberg-Steenrod theorem stating that for
a fibration $F\lra E \lra B$ we have equivalences
$$C_*(E)\simeq C_*(F)\tensor_{C_*(\Omega B)}k
\mbox{ and } C^*(E)\simeq \Hom_{C_*(\Omega B)}(k,C^*(F)).$$

\section{The Morita context}
\label{sec:Morita}
We have a simply connected rational space of finite type
$X$, and we consider the CDGA $C^*(X)$. We often wish to translate to
statements about the DGA $C_*(\Omega X)$. Throughout we work in
derived categories of DG-modules such as $\Der (C^*(X))$ or $\Der (C_*(\Omega X))$, so tensor products and Homs are derived. As mentioned above, we usually
refer simply to `modules' since the requirement that our modules respect the
differentials is  implicit in the category  we work in. The material is
adapted from \cite{tec,DGI1}.

\subsection{The two algebras.}

We need to see first that the $C^*(X)$ (a {\em commutative} DGA) and
 $C_*(\Omega X)$ (which will usually not be commutative) determine each other.
\begin{prop}
\label{prop:EMRS} If $X$ is 1-connected, there are equivalences
$$C_*(\Omega X) \simeq \Hom_{C^*(X)}(\Q,\Q)$$
and
$$C^*(X) \simeq \Hom_{C_*(\Omega X)}(\Q,\Q)$$
of DGAs.
\end{prop}

\begin{proof}
The first of these is the Eilenberg-Moore theorem \cite{EM}
and the second is the Rothenberg-Steenrod theorem \cite{RS}.
\end{proof}

\subsection{The adjunction.}
The proposition shows that we have an adjoint pair of functors
$$
\adjunction{\Hom_{C^*(X)}(\Q,\cdot)}
{\mbox{$C^*(X)$-mod}}
{\mbox{mod-$C_*(\Omega X)$}}
{(\cdot)\tensor_{C_*(\Omega X)}\Q}.$$
This induces an equivalence between subcategories of the derived
categories, but it will be enough for us to know we can move between
the module categories and to understand one composite.

\subsection{Cellularization.}

An object in the derived category of $C^*(X)$-modules is said
to be {\em $\Q$-cellular} if it is built from $\Q$ up to equivalence.
A map $M\lra N$ of $C^*(X)$-modules is a $\Q$-equivalence if
$$\Hom_{C^*(X)}(\Q,M)\lra \Hom_{C^*(X)}(\Q,N)$$
is a homology isomorphism. A map $M \lra N$ is  {\em $\Q$-cellular
approximation} if it is a $\Q$-equivalence and $M$ is $\Q$-cellular.
By the usual formal argument, this is unique up to equivalence,
and we write $\cell_\Q(N)\lra N$ for it.

We will give two models for $\Q$-cellularization, and it will be
valuable to know they are equivalent.

\subsection{The Morita model.} The first model comes from the
Morita context.

\begin{prop} \cite{tec, DGI1}
If $H^*(X)$ is Noetherian, the counit
$$\Hom_{C^*(X)}(\Q,M)\tensor_{C_*(\Omega X)}\Q \lra M$$
of the adjunction is $\Q$-cellularization. \qqed
\end{prop}

We need only  observe that $C^*(X)$ is proxy-regular in the sense of
\cite{DGI1}. Since $H^*(X)$ is Noetherian, the Koszul complex
associated to a system of parameters provides a proof.

\subsection{The stable Koszul model.}
If $R$ is a commutative ring and $I=(x_1, x_2, \ldots , x_r)$
is an ideal, then Grothendieck
defines the local cohomology of an $R$-module $N$  by the formula
$$H^*_I(R;N)=H^*((R\lra [\frac{1}{x_1}])\tensor_R
(R\lra [\frac{1}{x_2}])\tensor_R \cdots \tensor_R
(R\lra [\frac{1}{x_n}])\tensor_RN),  $$
and shows it calculates the right derived functors of $I$-power
torsion when $R$ is Noetherian. We write $H^*_I(R)=H^*_I(R;R)$ for brevity.

We now lift this to DGAs in the usual way.
If $x \in H^*(A)$, we write $\Gamma_xA=\fibre (A \lra A[1/x])$,
and if $I=(x_1,x_2, \ldots , x_n)$ is an ideal in $H^*(A)$, for
an $A$-module $M$ we write
$$\Gamma_I M=
\Gamma_{x_1}A\tensor_A \Gamma_{x_2}A \tensor_A
\cdots \tensor_A \Gamma_{x_n}A\tensor_A M. $$
It turns out that up to equivalence this depends only
on the ideal $I$, and indeed, only on the radical of $I$.
If $I=\frak{m}$ is the maximal ideal we abbreviate this
$\Gamma M=\Gamma_{\frak{m}}M$.

Note that $\Gamma M$ has a filtration from its construction, and
that we therefore have a spectral sequence for calculating its
homology.
\begin{lemma}
There is a spectral sequence
$$\lc^*(H^*(A);H^*(M))\Rightarrow H^*(\Gamma M).\qqed$$
\end{lemma}

Finally, the relevance to us is that this gives another construction
of cellularization.

\begin{prop}
\cite[9.3]{DGI1}
The natural map
$$\Gamma M \lra M$$
is $\Q$-cellularization.\qqed
\end{prop}

Now we specialize to the case $A=C^*(X)$ to obtain the required equivalence
from uniqueness of cellularization.
\begin{cor}
\label{cor:GammaisCellk}
There is a natural equivalence
$$\Gamma M \simeq \Hom_{C^*(X)}(\Q,M)\tensor_{C_*(\Omega X)}\Q .\qqed$$
\end{cor}

\section{Regular rings and spaces.}
\label{sec:regular}
We shall show that the regular spaces are precisely the
spaces $KV$ where $V$ is even and finite dimensional.
This is straightforward once we have established definitions.

For all classical commutative algebra, we refer the reader to
\cite{Matsumura}.

\subsection{Definitions.}
In commutative algebra there are three styles for a definition of
a regular local ring: ideal theoretic, in terms of the {\bf g}rowth of
the Ext algebra and a {\bf h}omotopy invariant version.

\begin{defn}
(i) A local Noetherian ring $R$ is {\em regular} if the maximal ideal
is generated by a regular sequence.

(ii) A local Noetherian ring $R$ is {\em g-regular} if $\Ext_R^*(k,k)$ is
finite dimensional.

(iii) A local Noetherian ring $R$ is {\em h-regular} if every finitely generated
module is small in $\Der (R)$.
\end{defn}

It is not hard to see that g-regularity is equivalent to h-regularity or that
regularity implies g-regularity. Serre proved that g-regularity implies regularity, so the three conditions are equivalent.

It is not altogether clear what should play the role of finitely generated
modules in the more general context. We would like it to include all
small objects, and the object $\Q$, and we would like to know that
if $\Q$ is small then all objects in the class are small. For the purpose
of the present paper, we take
$$\fg := \{ M \st H^*(M) \mbox{ is a finitely generated $H^*(X)$-module} \},$$
and we will show that it has the properties we require.

\begin{defn}
(i)  A space $X$  is {\em s-regular} if there are fibrations
$$S^{n_1}\lra X_1 \lra X, S^{n_2}\lra X_2 \lra X_1,
\ldots , S^{n_d}\lra X_d \lra X_{d-1}$$
with $X_d\simeq *$.

(ii) A space $X$  is {\em g-regular} if $H_*(\Omega X)$ is finite
dimensional.

(iii) A space $X$ is {\em h-regular} if every object of $\fg$
 is small in $\Der (C^*(X))$.
\end{defn}

If $X$ is s-regular, we see $\Omega X_{d-1}\simeq S^{n_d}$,
and working back up the sequence of fibrations, we see that
$X$ is g-regular.
Since $\Q \in \fg$ it follows from Proposition \ref{prop:EMRS}
 that an h-regular space is
g-regular. We will establish the reverse implication by classifying
g-regular spaces.

\begin{remark} The use of the classification
is somewhat unsatisfactory, and suggests that
we should seek a choice of class $\fg$ that is appropriate even when
we do not have such a classification. One possibility is to consider
all $R$-modules $M$ which are small as $Q$-modules, for some map $Q\lra R$
of algebras from a regular ring $Q$ so that $R$ is small as a $Q$-module.
\end{remark}

\subsection{Classification of regular spaces.}

In the rational context we can give a complete classification of
regular spaces.

\begin{thm}
\label{thm:hregisKV}
A simply connected rational space $X$ of finite type is g-regular if and only
if $\pi_*(X)$ is even and finite dimensional. It is therefore
equivalent to the Eilenberg-MacLane space $K(\pi_*(X))$, and has polynomial
cohomology $\symm (\pi_*(X))$.
\end{thm}

\begin{proof}
Since $\Omega X$ is a product of Eilenberg-MacLane spaces,
we need only remark that odd Eilenberg-MacLane spaces are spheres,
whereas even Eilenberg-MacLane spaces are infinite dimensional.
\end{proof}

\begin{prop}
\label{prop:fgregissmall}
If $X$ is g-regular and $H^*(M)$ is finitely generated over $H^*(X)$
then $M$ is small.
\end{prop}

\begin{proof}
Suppose $H^*(M)$ is a finitely generated $H^*(X)$-module. Since
$H^*(X)$ is a polynomial ring on even degree generators, there is a
finite resolution by finitely generated free modules
$$0\lra P_r\stackrel{d_r}\lra P_{r-1}\stackrel{d_{r-1}}\lra \cdots
\stackrel{d_2}\lra P_1\stackrel{d_1}\lra P_0 \stackrel{d_0}\lra H^*(M)\lra 0.$$
We proceed to realize this in the usual way. To start with we realize the
free modules $P_i=(H^*(X))^{\oplus n}$ by the $C^*(X)$-modules
$\bbP_i=(C^*(X))^{\oplus n}$. Now take $M=M_0$ and realize the algebraic
resolution by constructing a diagram
$$\diagram
\bbP_0 \dto & \Sigma \bbP_1\dto  &&  \Sigma^r\bbP_r\dto& &\\
M_0 \rto & M_1\rto  &\cdots \rto & M_r\rto &M_{r+1}\simeq 0\\
 \enddiagram$$
in which the sequences $\Sigma^i \bbP_i \lra M_i \lra M_{i+1}$ are
cofibre sequences,
$H^*(\Sigma^{-i}M_i)=\ker (d_{i-1})$ and $\Sigma^i\bbP_i \lra M_i$ realizes
the map in the algebraic resolution. Reversing the process, we
see that $M_r,M_{r-1}, \ldots, M_1 $ and $M_0=M$ are finitely built from
$C^*(X)$ and therefore small as required.
\end{proof}

This establishes the equivalence of the two definitions of regularity.
\begin{cor}
A space is g-regular if and only if it is h-regular.\qqed
\end{cor}

\subsection{Some small objects.}
It is useful to identify some modules that are small rather generally.

\begin{lemma}
\label{lem:finfibissmall}
If $f: Y \lra X$ is a map with homotopy fibre $F(f)$ so that
$H_*(F(f))$ is finite
dimensional, then $C^*(Y)$ is small in  $\Der (C^*(X))$.
\end{lemma}

\begin{proof}
By hypothesis,  $\Q$ finitely builds $ C_*(F(f))$ as a
$C_*(\Omega X)$-module. Applying
$\Hom_{C_*(\Omega X)}( \Q, \cdot )$, we deduce from the Eilenberg-Moore
spectral sequence that $C^*(X)$ finitely builds $C^*(Y)$. In symbols,
$$\Q \finbuilds_{C_*(\Omega X)}  C_*(F(f))$$
and hence
$$C^*(X) \simeq  \Hom_{C_*(\Omega X)}(\Q,\Q) \finbuilds
\Hom_{C_*(\Omega X)}(\Q, C_*(F(f)) \simeq C^*(Y) .$$
\end{proof}

\begin{lemma}
\label{lem:regimpliesfgreg}
If $X$ is g-regular, and $Y \lra X$ is a map with $C_*(F(f))$ finitely
built from $C_*(\Omega X)$ then $C^*(Y)$ is small.
\end{lemma}

\begin{proof}
Suppose $X$ is
g-regular, so that  $H_*(\Omega X)$ is finite dimensional. Thus $\Q$ finitely
builds $C_*(\Omega X)$. It follows that
if  $C_*(\Omega X)$ finitely builds $C_*(F(f))$ then $\Q$ finitely builds
$C_*(F(f))$ and we may apply the argument of Lemma \ref{lem:finfibissmall}.
\end{proof}






\section{The centre of a triangulated category.}
\label{sec:centre}
It will be useful to recall certain constructions before turning to
complete intersections.

\subsection{Universal Koszul complexes.}
To start with we suppose given a triangulated category $\cT$. The
{\em centre} $Z\cT$ of $\cT$ is defined to be the graded ring of
graded endomorphisms of the identity functor.

Given $\chi \in Z\cT$ of degree $a$,  for any object $X$, we may form
the mapping cone $X/\chi$ of $\chi : \Sigma^a X \lra X$. This is
well defined up to non-unique equivalence. Indeed, given a map
$f: X \lra Y$, the axioms of a triangulated category give
 a map $f:X/\chi \lra Y/\chi$ consistent with
the defining triangles, but this is not usually unique or compatible
with composition.

Now given a sequence of elements $\chi_1 , \chi_2, \ldots, \chi_n$
we may iterate this construction, and form
$$K(X; \chivec) := X/\chi_1/\chi_2/\cdots /\chi_n, $$
which we refer to as the universal Koszul complex of the sequence.
Once again, up to equivalence $K(X;\chivec)$ depends only on the sequence, and
is independent of the order of the
elements $\chi_i$.

\subsection{Bimodules and the centre.}
 Bimodules provide a useful source of elements of
$Z\Der (R)$. Indeed, if $R$ is a flat $l$-algebra, and if $X \lra Y$ is a
map of  $R$-bimodules
over $l$ (which is to say, of modules over $R^e=R \tensor_l R$), then
for any $R$-module $M$ we obtain a map
$$X \tensor_R M \lra Y \tensor_R M$$
of $R$-modules, natural in $M$.

It is sometimes convenient to package this in terms of the Hochschild
cohomology ring
$$HH^*(R|l)=\Ext_{R^e}^*(R,R). $$
If $l=\Z$, it is usual to omit it from the notation.
Now a codegree $d$ element of this cohomology ring can be viewed as a map
$R\lra \Sigma^d R$ in the category of $(R,R)$-bimodules, so that
taking $X=Y=R$ above, we obtain a ring homomorphism
$$HH^*(R|l) \lra Z\Der (R). $$
If $R$ is an $l$-algebra which is not flat, $R^e=R \tensor_l R$ is
taken in the derived sense, and similarly for $HH^*(R|l)$.

Given maps
$l\lra Q \lra R$, we obtain a map $R\tensor_lR \lra R\tensor_{Q}R$
and hence a ring map $HH^*(R|Q)\lra HH^*(R|l)$. In particular, we have
maps
$$R=HH^*(R|R) \lra HH^*(R|Q) \lra HH^*(R|\Z)=HH^*(R).$$

If $R=C^*(X)$, we may always take $l=\Q=C^*(pt)$, so that a bimodule is
a module over $R^e=C^*(X \times X)$, but it is usually more appropriate to
work over $Q=C^*(K)$ where we have a fibration $X \lra K$. In that case
a bimodule over $Q$ is a module over $R^e=C^*(X\times_K X)$.

\subsection{Hochschild cohomology transcended.}
\label{subsec:transcended}

It seems natural to relax the role of  Hochschild cohomology.
For us it is really just a tool for building bimodules from $R$.
We will suppose given a map $Q\lra R$ so that $Q$ is regular
and $R$ is small over $Q$. This ensures that $\fg $ as
defined in Section \ref{sec:regular} coincides with the $R$-modules
which are small over $Q$.

Now, if $X$ is any $R$-bimodule finitely built from $R$, we
can apply $\tensor_R M$ to deduce $X\tensor_R M$ is finitely
built from $M=R\tensor_RM$:
$$R\finbuilds_{R^e} X \mbox{ implies }
 M=R\tensor_R M \finbuilds_R X \tensor_R M.$$
The important case for us is when $X$ is a small $R^e$-module.
\begin{lemma}
\label{lem:bimodisallmod}
If $X$ is a small $R^e$-module and $M\in \fg$, then
$X\tensor_R M$ is a small $R$-module.
\end{lemma}
\begin{proof}
It suffices to consider the case $X=R^e$. We then have
$$X\tensor_R M=R\tensor_QR \tensor_R M=R\tensor_Q M.$$
By Proposition \ref{prop:fgregissmall}, $M$ is small as a $Q$-module, so
$$R=R\tensor_QQ\finbuilds R\tensor_QM$$
as required.
\end{proof}

This comes close to saying that if $R$ is virtually small as an
$R^e$-module then every $M\in \fg$ is virtually small as an $R$-module.
The only obstacle is the need to show $X\tensor_RM$ is non-trivial; in the
context we need it, the non-zero degree of the maps constructing $X$
will make it clear.

\section{Complete intersection rings and spaces.}
\label{sec:ci}

We will give definitions of complete intersections as in the regular case.
For commutative Noetherian rings these were shown to be equivalent in
\cite{BGzci}. We will show they are equivalent for rational spaces.

\subsection{The definition}

In commutative algebra there are three styles for a definition of
a complete intersection ring: ideal theoretic, in terms of the growth of
the Ext algebra and a derived version.

\begin{defn}
(i) A local Noetherian ring $R$ is a {\em complete intersection (ci)}
ring if $R=Q/(f_1,f_2, \ldots , f_c)$ for
some regular ring $Q$ and some regular sequence $f_1, f_2, \ldots , f_c$.
The minimum such $c$ (over all $Q$ and regular sequences) is called
the {\em codimension} of $R$.

(ii) A local Noetherian ring $R$ is  {\em gci}
if $\Ext_R^*(k,k)$ has polynomial growth. The  {\em g-codimension}
of $R$ is one more than the degree of the growth.

(iii) A local Noetherian ring $R$ is  {\em zci} \cite{BGzci} if
there are elements $z_1,z_2,\ldots z_c\in Z\Der (R)$ of non-zero degree so that
$M/z_1/z_2/ \cdots /z_c$ is small for all finitely generated modules $M$.
The minimum such $c$ is called the  {\em z-codimension} of $R$.

 {
The zci condition implies that every finitely generated module finitely builds a
small complex in a prescribed manner using elements in $Z\Der (R)$. We can relax
this by demanding only that each step in the building of the small
complex is the cone of an \emph{endomorphism} of the previous step. This is
the essence of the next definition.

(iv) A local Noetherian ring $R$ is  {\em eci} if
there is a regular ring $Q$, a map $Q\lra R$ and homotopy cofibration sequences
of $R^e$-modules, where $R^e=R\tensor_QR$,
\[ R=M_0 \xrightarrow{g_1} \Sigma^{n_1} M_0 \to M_1,\ \ldots \ ,\ M_{c-1}
\xrightarrow{g_c} \Sigma^{n_c} M_{c-1} \to M_c\]
such that $M_c$ is small as an $R^e$-module and the degree of each $g_i$ is
not zero.
}

Two variations are also useful.

(v) A local Noetherian ring $R$ is said to be  {\em bci}
if there is a regular ring $Q$ and map $Q\lra R$ so that
$R$ is virtually small as an $R^e$-module, where $R^e=R\tensor_QR$.

(vi) If $R$ is a commutative ring or CDGA, it is said to be a
{\em quasi-complete intersection (qci)} \cite{DGI2}
if every finitely generated object  is virtually small.
\end{defn}

\begin{thm} \cite{BGzci}
\label{thm:BGzciTheorem}
For a local Noetherian ring the conditions ci, gci and zci are all
equivalent, and the corresponding codimensions are equal.
 These conditions imply the  {eci}, bci and qci conditions.
\end{thm}

It is a result of Gulliksen that if $R$ is ci of codimension
$c$, one may construct a resolution of any finitely generated module growing
like a polynomial of degree $c-1$. A suitable construction of this
resolution shows that  $R$ is zci. Considering the module $k$
shows that the ring $\Ext_R^*(k,k)$ has polynomial
{\bf g}rowth. Perhaps the most striking result about
ci rings is the theorem of Gulliksen
\cite{Gulliksen} which states that this characterises ci rings so that the
ci and gci conditions are equivalent for local rings.

\begin{remark}
In commutative algebra, Avramov \cite{AvramovAQ}
proved Quillen's conjectured characterization
of complete intersections by the fact that the Andr\'e-Quillen cohomology
is bounded. When $k$ is of characteristic 0, the DG Andr\'e-Quillen cohomology
of $C^*(X)$ gives the dual homotopy groups of $X$, so the counterpart of
 Avramov's characterization is the gci condition.

On the other hand in positive characteristic, results of Mandell
\cite{Mandell}
show that the topological Andr\'e-Quillen cohomology of $C^*(X)$
vanishes quite
generally, so this does not give an appropriate counterpart of the ci
condition.
\end{remark}

\subsection{Definitions for spaces.}
Adapting the above definitions for spaces is straightforward.

\begin{defn}
(i) A space $X$ is {\em spherically ci (sci)} if it is formed from a regular space $KV$
using a finite number of spherical fibrations. More precisely, we require that
there is a regular space $X_0=KV$ with $V$ even and finite dimensional, and
fibrations
$$S^{n_1}\lra X_1 \lra X_0=KV, S^{n_2}\lra X_2 \lra X_1, \ldots ,
S^{n_c}\lra X_c \lra X_{c-1}$$
with $X=X_c$. The least such $c$ is called the {\em s-codimension} of $X$.

(ii) A space  $X$ is a {\em gci} space
if $H^*(X)$ is Noetherian and $H_*(\Omega X)$ has polynomial growth.
The {\em g-codimension} of $X$ is one more than the degree of growth.

(iii) A space $X$ is a {\em zci} space if $H^*(X)$ is Noetherian and
there are elements $z_1, z_2, \ldots , z_c\in Z \Der (C^*(X))$ of non-zero degree
so that $C^*(Y)/z_1/z_2/ \cdots / z_c$ is small for all
$C^*(Y) \in \fg$.

 {
(iv) A space $X$ is an {\em eci} space if $H^*(X)$ is Noetherian,
there is a regular space $K$ and fibration $X\lra K$ with $C^*(X)$ small over $C^*(K)$
and there are homotopy cofibration sequences of $C^*(X\times_KX)$-modules,
\[ C^*(X)=M_0 \xrightarrow{g_1} \Sigma^{n_1} M_0 \to M_1,\ \ldots \ ,\ M_{c-1}
\xrightarrow{g_c} \Sigma^{n_c} M_{c-1} \to M_c\]
such that $M_c$ is small as an $C^*(X\times_KX)$-module and the degree of each $g_i$ is
not zero.
}

(v) We say $X$ is  {\em bci} space if $H^*(X)$ is Noetherian and
$C^*(X)$ is virtually small as a $C^*(X\times_KX)$-module for some
regular space $K$ and fibration $X\lra K$ with $C^*(X)$ small over $C^*(K)$.

(vi) We say $X$ is  {\em qci} space if $H^*(X)$ is Noetherian and
 each $C^*(Y) \in \fg$ is virtually small.
\end{defn}

The main result of this paper is as follows.

\begin{thm}
\label{thm:equivalence zcigcieci}
For a rational space $X$ the sci,  {eci } and gci
conditions are equivalent.  {If in addition $X$ is pure, then the conditions
above are equivalent to the zci condition.}
\end{thm}

We will establish the implications
$$sci \stackrel{A}\Rightarrow  {eci} \stackrel{B}\Rightarrow gci
\stackrel{C}\Rightarrow sci. $$
We establish A in Section \ref{sec:sciisbci}, B in Section \ref{sec:zciisgci},
and C in Section \ref{sec:gciissci}. The first two implications are fairly
straightforward in the sense that they can also be proved in the non-rational
context \cite{pzci}. The implication C takes a growth condition and
gives a structure theorem, and could be viewed as
the main result of the present paper.
 {In Section \ref{sec:bciiszci} we show that a pure sci space is zci,
while in Section \ref{sec:zciisgci} we show that the zci condition implies gci.}

\begin{remark}
(i) If $X$ is elliptic then $H^*(X)$ and $\pi_*(\Omega X)$ are both
 finite dimensional, so it is clear that every elliptic space is gci.

(ii) It is also clear that zci implies qci, and that if the natural
transformations giving the zci condition come from Hochschild cohomology
then this implies  {eci, and eci clearly implies } bci.
\end{remark}

\subsection{Hypersurface rings.}

A hypersurface is a complete intersection of codimension 1. The first
four definitions adapt to define hypersurfaces, 
g-hypersurfaces,  {z-hypersurfaces and e-hypersurfaces}.
The notion of g-hypersurface (i.e., the dimension of
the groups $\Ext^i_R(k,k)$ is bounded) may be strengthened to the
notion of p-hypersurface where we require
that they are eventually periodic, given by multiplication with an element
of the ring. All  {five } of these conditions are equivalent
by results of Avramov.

One possible formulation of b-hypersurface would be to require that
the $R$ builds a small $R^e$-module in one step (or equivalently,
that $R$ is a
z-hypersurface but $z$ arises from $HH^*(R)$).
 {Both these definitions are equivalent to being an e-hypersurface.}

Finally, we may say that $R$ is a q-hypersurface if every finitely
generated module $M$ has a self map with non-trivial small mapping cone.

\subsection{Hypersurface spaces.}

All six of these conditions have obvious formulations for spaces.
A space $X$ is an {\em s-hypersurface} if there is a fibration
$$S^n \lra X \lra KV$$
with $V$ even and finite dimensional.  It is a {\em z-hypersurface}
if there is an element $z$ of non-zero degree in $Z\Der (C^*(X))$ so that,
for any $M$ in $\fg$, the mapping cone of $z:M\lra M$ is small.
 It is a {\em g-hypersurface}
if the dimensions of  $H_i(\Omega X)$ are bounded, and a {\em p-hypersurace}
if they are eventually periodic given by multiplication by an element of the
ring.

 {The space $X$ is an {\em e-hypersurface} if $C^*(X)$  builds a small
$C^*(X\times_KX)$-module in one step for a regular space $K$. }
Finally,  $X$ is a {\em q-hypersurface} if every finitely
module $C^*(Y)$ in $\fg$ has a self map with non-trivial small mapping cone.

\section{Standard form for sci spaces}
\label{sec:sci}

We are eventually going to show that the sci, gci and  {eci }
conditions are equivalent for rational spaces. Of the
conditions, the easiest to get a grip on is the sci condition, and it seems
worthwhile to begin by anchoring it  in reality by giving a structure theorem.
In the rational context, we may put sci spaces into a standard form.

\begin{thm}
\label{thm:sci}
A space $X$ is sci if and only if there exists a fibration sequence
\[ F \to X \to KV , \]
where $KV$ is a regular space and $\pi_*(F)$ is finite dimensional and
entirely in odd degrees; in this case
$$\codim (X)=\dimQ (\pi_*(F))=\dimQ (\pi_{odd}(X)).$$
\end{thm}

Before proceeding it is useful to note that all the spherical fibrations
in the definition of an sci space may be taken to be odd.

\begin{lemma}
\label{lem:sciodd}
If $X$ can be formed from $B$ with an even spherical
fibration $S^{2m} \lra X \lra B$, then it can be formed
from $B\times K(\Q ,2m)$ by an odd spherical fibration
$$S^{4m-1} \lra X \lra B\times K(\Q,2m).$$
Accordingly, an sci space of codimension $c$ may be constructed in $c$ steps
from  a regular space using only odd dimensional spherical fibrations.
\end{lemma}

\begin{proof}
If $C^*(X)=C^*(B)\sdr \Lambda (x_{2m},y_{4m-1})$, then if
$dy=x^2+ax+b$ we may change basis by taking
$x'=x+a/2$ and find $dx'=0, dy=(x')^2+z$, where
$z=b-a^2/4\in C^*(B)$.
Adjoining $x'$ to the model of $B$, we get the base of the
required fibration.
\end{proof}

\begin{proof}[of Theorem \ref{thm:sci}]
If $X$ is sci, by Lemma \ref{lem:sciodd}
we may use only odd spheres in the fibres. Now the composite function
$X \lra KV$ has fibre with only odd dimensional homotopy, giving a fibration
of the stated form.

We prove the converse statement by induction on the dimension of the
odd homotopy. The result  is trivial if the homotopy is entirely even. Suppose
then that  $X$ lies in a fibration
$$F \lra X \lra KV$$
and that  $x \in \pi_m(F)$ is an element of highest degree.
Construct a fibration
$$ S^m \lra F \lra F'$$
by killing  $x$, so that $\dimQ (\pi_*(F'))=\dimQ (\pi_*(F))-1$.
Thus, we may choose models so that
$C^*(F)=C^*(F')\sdr C^*(S^m)$, and
$$C^*(X)=C^*(KV)\sdr [C^*(F')\sdr C^*(S^m)]. $$
Let $X'$ be modelled by the subalgebra generated by $C^*(KV)$ and
$C^*(F')$. This gives  fibrations
$$S^m\lra X\lra X' \mbox{ and } F'\lra X'\lra KV.$$
By induction $X'$ is sci,  so that $X$ is sci as required. The
codimension is obviously bounded below by $\dimQ (\pi_{odd}(X))$, and
we have described a procedure achieving this bound.
\end{proof}

The following rearrangement result will be useful later.

\begin{cor}
\label{cor:scirearrangement}
If $X$ occurs in a fibration
$$X' \lra X\lra KV$$
with $X'$ sci of codimension $c$, then $X$ is itself
sci of codimension $c$.
\end{cor}

\begin{proof}
By Theorem \ref{thm:sci}, $X'$ has a model of the
form $X'=KV' \sdr F'$ with $\pi_*(F')$ finite dimensional
and in odd degrees and $X=KV \sdr X'$
with both $V$ and $V'$ even and finite dimensional.
By parity there can be no differential from $KV$ to $KV'$, so
$$X=KV' \sdr (KV\sdr F')\simeq (KV' \sdr KV)\sdr F'.$$
Since any fibration with base $KV'$ and fibre $KV$ is a
product, we obtain a fibration
$$F'\lra X\lra K(V\oplus V'). $$
By Theorem \ref{thm:sci} again we deduce $X$ is sci of codimension $c$.
\end{proof}

In terms of rational models we can restate the sci condition very simply.
The result is immediate from Theorem \ref{thm:sci} by taking a Sullivan
model of the fibration.
\begin{cor}
A space $X$ is sci if and only if $X$ has a cochain algebra model $(\Lambda V,d)$ where $d(V^{even})=0$. \qqed
\end{cor}

\section{Growth conditions.}
\label{sec:zciisgci}
In this section we prove perhaps the simplest implication between the ci
conditions: for simply connected rational spaces of finite type, 
 {eci (and also zci) implies gci}.

\subsection{Polynomial growth.}
Throughout algebra and topology it is common to use the rate of growth of
homology groups as a measurement of complexity. We will be working over
$H^*(X)$, so it is natural to assume that our modules $M$
are  {\em locally finite} in the sense that  $H^*(M)$ is cohomologically
 bounded below  and
 $\dimQ (H^i(M))$ is finite for all $i$.

\begin{defn}
We say that a locally finite module $M$ has {\em polynomial growth of
degree $\leq d$},  and write
$\growth (M) \leq d$, if
there is a polynomial $p(x)$ of degree $d$ with
$$\dimQ (H^{n}(M))\leq p(n)$$
for all $n >> 0$.
\end{defn}

\begin{remark}
(i) In commutative algebra the usual terminology is that a
 module of growth $d$ has {\em complexity} $d+1$.

(ii) Note that a complex with bounded homology has growth $\leq -1$. For
complexes with growth $\leq d$ with $d\geq 0$, by adding a constant to the
polynomial, we may insist that the
bound applies for all $n \geq 0$.
\end{remark}

\subsection{Mapping cones reduce degree by one.}
We use the following estimate on growth.
\begin{lemma}
\label{lem:growthintriangles}
Given cohomologically bounded below locally finite modules $M$ and $N$
in a triangle
$$\Sigma_n M \stackrel{\chi}\lra M \lra N$$
with $n \neq 0$, then
$$\growth (M) \leq \growth (N)+1.$$
\end{lemma}

\begin{proof}
The homology long exact sequence of the triangle includes
$$\cdots \lra H^{i-n}(M) \stackrel{\chi}\lra
H^i(M) \lra H^i(N)\lra \cdots .$$
This shows
$$\dimQ (H^i(M))\leq \dimQ (H^i(N))+\dimQ (\chi H^{i-n}(M)) .$$
Iterating $s$ times, we find
\begin{multline*}
\dimQ (H^i(M))\leq \dimQ (H^i(N))+\dimQ(H^{i-n}(N)) +\cdots
\\ \cdots +
\dimQ(H^{i-(s-1)n}(N)) +
\dimQ(\chi^s H^{i-sn}(M)) .
\end{multline*}

To obtain growth estimates, it is convenient to collect
the dimensions of the homogeneous parts into the Hilbert
series $h_M(t)=\sum_n\dimQ (H^i(M))t^i$. An inequality
between such formal series means that it holds between all
coefficients.

First suppose that $n>0$. Since  $H^*(M)$ is bounded below,
if $h_M(t)$ is the Hilbert series of $H^*(M)$ then we have
$$h_M(t) \leq h_N(t)(1+t^{n}+t^{2n}+\cdots )=\frac{h_N(t)}{1-t^{n}},$$
giving the required growth estimate.

If $n=-n'<0$ we rearrange to obtain
$$N'\lra M \lra \Sigma_{n'}M$$
where $N'=\Sigma_{n'-1}N$ and argue precisely similarly.
\end{proof}

\subsection{Growth of  {eci } spaces.}
The implication we require is now straightforward.

 {
\begin{thm}
If $X$ is eci  then it is also gci, and if $X$ has e-codimension $c$
it has g-codimension $\leq c$.
\end{thm}

\begin{proof}
It is sufficient to show $C^*(\Omega X)\simeq \Q \otimes_{C^*(X)}\Q $ has
polynomial growth.

By hypothesis there is an appropriate regular space $K$ and self maps
\[ \gamma_1:M_0 \to \Sigma_{|\gamma_1|} M_0,\
\gamma_2:M_1 \to \Sigma_{|\gamma_1|} M_1,\ \ldots \ , \
\gamma_c: M_{c-1} \to \Sigma_{|\gamma_c|} M_{c-1}\]
of non-zero degree in $\Der (C^*(X\times_KX))$,
so that $M_i$ is the cone of $\gamma_i$ and $M_c$, which is the cone of 
$\gamma_c$, is small.
Thus, applying $\Q \otimes_{C^*(X)}(\cdot)$ to $M_c$ we obtain a
complex with  growth $\leq -1$.
By the lemma if we apply  $\Q \otimes_{C^*(X)}(\cdot)$ to
$M_{c-1}$ we obtain a complex
 of growth $\leq 0$. Doing this repeatedly, we  deduce that when we
apply $\Q \otimes_{C^*(X)}(\cdot)$ to $\Q$ itself we obtain a complex
with  growth $\leq c-1$ as required.
\end{proof}
}

 {The proof above, with minor changes, also yields the following Theorem.}
\begin{thm}
If $X$ is zci  then it is also gci, and if $X$ has z-codimension $c$
it has g-codimension $\leq c$.
\end{thm}

\section{sci spaces are  {eci } spaces}
\label{sec:sciisbci}

In this section we show that sci spaces (defined by a particular construction)
have a periodic module theory in the sense that they are  {eci}.
This may not
be too surprising, but the particular way in which bimodules and fibrations
are used may be of some interest.

\begin{thm}
\label{thm:sciisbci}
If $X$ is an sci space of codimension $c$, then it is  {eci } of codimension $c$.
\end{thm}
\begin{remark}
\label{rem:bimodisallmod}
The construction will show that all the maps building the small bimodule
are of positive degree, so that Lemma \ref{lem:bimodisallmod} shows that
if $X$ is sci then all
$C^*(X)$-modules in $\fg$ are virtually small.
\end{remark}

We will upgrade the conclusion to show  {that if $X$ is a pure sci space
then $X$ } is zci in Section \ref{sec:bciiszci}.

\subsection{Fibration lemmas}
We will repeatedly use two elementary  lemmas. The first is very well known.
\begin{lemma}
\label{lem:fib1}
If
$$F\lra E \stackrel{p}\lra B$$
is a fibration with a section $s$, then there is a fibration
$$\Omega F\lra B \stackrel{s}\lra E. $$
\end{lemma}

\begin{proof} We start from the square
$$\diagram
B \rto \dto^= & E\dto^p\\
B\rto^= &  B
\enddiagram$$
and take iterated fibres.
\end{proof}

The second lemma is a Third Isomorphism Theorem for fibrations.
\begin{lemma}
\label{lem:fib2}
Given fibrations $Y\lra B \lra C$, if $F=\fibre (B \lra C)$
there is a fibration
$$\Omega F \lra Y\times_BY \lra Y\times_CY.$$
\end{lemma}

\begin{proof}
We start from the cube
$$\diagram
&Y\times_C Y \rrto \ddto &&Y\ddto \\
Y\times_B Y \rrto \ddto \urto && Y\ddto \urto &\\
&Y \rrto  &&C \\
Y\rrto \urto &&B\urto&
\enddiagram$$
and take iterated fibres.
\end{proof}

\subsection{Building bimodules.}
It is worth isolating the process that we use repeatedly
to build bimodules. Abstracted from its context, the proof is
extremely simple. The strength of the result is that the
cofibre sequence is one of $C^*(Y)$-modules.

\begin{prop}
\label{prop:loopspherebuild}
Suppose given a  fibration
$$\Omega S^m \lra X \stackrel{f}\lra Y.$$

(i) If $m$ is  odd,
then there is a cofibre sequence of $C^*(Y)$-modules
$$\Sigma_{m-1} C^*(X) \lla C^*(X) \stackrel{f^*}\lla C^*(Y). $$

(ii) If $m$ is  even,
then there is a cofibre sequence of $C^*(Y)$-modules
$$\Sigma_{2m-2} C^*(X) \lla C^*(X) \lla F $$
with $F$ small. More precisely, $F$ is built from two copies of $C^*(Y)$
in the sense that there is a cofibre sequence
$$\Sigma_{m-1} C^*(Y) \lla F \lla C^*(Y)  $$
of $C^*(Y)$-modules.
\end{prop}

\begin{proof}
We take a relative Sullivan model for $f^*$.
This has the form $C^*(Y)\sdr C^*(\Omega S^m)$.

If $m$ is odd,  the relative Sullivan model is of the form
$C^*(X)\simeq C^*(Y)\sdr \Lambda (z_{m-1})$,
where $z$ is a polynomial generator. The quotient of this by
the DG-$C^*(Y)$-submodule
$C^*(Y)\cdot z^0$ is again a $C^*(Y)$-module, and the composite
$$\Sigma_{m-1}C^*(X)\stackrel{z}\lra C^*(X) \lra C^*(X)/C^*(Y)$$
is an isomorphism as required.

If $m$ is even,  the relative Sullivan model is of the form
$$C^*(X)\simeq C^*(Y)\sdr \Lambda (z_{m-1},t_{2m-2}).$$
This time  $t$ is a polynomial generator,  and $z$ is an
exterior generator. Accordingly we let $F$ be the
DG-$C^*(Y)$-submodule generated by $t^0$ and $z$, so that
$C^*(Y)/F$ is again a $C^*(Y)$-module.
It is isomorphic to $\Sigma_{2m-2}C^*(X)$
 in the sense that the composite
$$\Sigma_{2m-2}C^*(X)\stackrel{t}\lra C^*(X) \lra C^*(X)/F$$
is an isomorphism.
\end{proof}

\begin{remark}
As an example, we observe that this shows that an s-hypersurface is
a z-hypersurface. Indeed, by hypothesis, we have a fibration
$S^m\lra X \lra KV$, and hence by pullback a fibration
$$S^m \lra X\times_{KV} X \lra X$$
with a section $\Delta : X \lra X\times_{KV}X $. By
Lemma \ref{lem:fib1},  we obtain a fibration
$$\Omega S^m \lra X \stackrel{\Delta}\lra X\times_{KV}X$$
so we may apply Proposition \ref{prop:loopspherebuild} with $Y=X\times_{KV}X$, noting
that a $C^*(Y)$-module is then a $C^*(X)$-bimodule. If $m$ is odd
we then get a cofibre sequence
$$\Sigma_{m-1}C^*(X) \stackrel{\tilde{\chi}}\lla C^*(X)\lla C^*(X\times_{KV}X)$$
of bimodules. As in Subsection \ref{subsec:transcended} note that $\tilde{\chi}$
gives an element of $Z\Der (C^*(X))$ by tensoring down, in the sense that for
any $C^*(X)$-module $M$ we apply $M\tensor_{C^*(X)}(\cdot)$ to get
a cofibre sequence
$$\Sigma_{m-1}M \stackrel{\chi}\lla M \lla C^*(X)\tensor_{C^*(KV)}M. $$
If $M$ is finitely generated,  then it is small as a $C^*(KV)$-module
by Proposition \ref{prop:fgregissmall}
showing that the fibre of $\chi$ is small as required.

The argument when $m$ is even is precisely similar.
\end{remark}

\subsection{The proof}

We now have the necessary ingredients for proving Theorem  {\ref{thm:sciisbci}}.


We suppose $X$ is sci of codimension $c$, so that we may form
$X=X_c$ in $c$ steps from $X_0=KV$ using fibrations
$$S^{n_i}\lra X_i \lra X_{i-1}.  $$
It will simplify the argument to assume all the spheres are
odd dimensional, as we may do by Lemma \ref{lem:sciodd}.

We must show that $C^*(X)$ builds $C^*(X\times_{KV}X)$ as a bimodule
(i.e., as a $C^*(X\times_{KV}X)$-module) using $s$ cofibre sequences.
It is convenient to write $X^e_i=X\times_{X_i}X$, and $X^e=X^e_0$,
so that we want to work with $C^*(X^e)$-modules. However, since we have maps
$$X=X_s\lra X_{s-1} \lra \cdots \lra X_0=KV, $$
we have maps
$$X_s=X^e_s\lra X^e_{s-1} \lra \cdots \lra X^e_0 =X^e, $$
so we may view $C^*(X^e_i)$-modules as $C^*(X^e_0)$-modules by restriction.

We are ready to apply our fibration lemmas.

Pulling back the fibration along $X_i\lra X_{i-1}$ we obtain a
fibration
$$S^{n_i}\lra X_i\times_{X_{i-1}}X_i \stackrel{\pi_1}\lra X_{i}  $$
with a section given by the diagonal $\Delta$. Applying Lemma
\ref{lem:fib1} we obtain a fibration
$$\Omega S^{n_i}\lra X_i \stackrel{\Delta}\lra  X_i\times_{X_{i-1}}X_i.  $$

Similarly, applying Lemma \ref{lem:fib2} to $X_s\lra X_i\lra X_{i-1}$
where $s\geq i$, we obtain a fibration
$$\Omega S^{n_i}\lra X^e_i  \lra  X^e_{i-1}.  $$

Now using the first of these, Proposition \ref{prop:loopspherebuild}
gives a cofibration
$$\Sigma_{n_s}C^*(X) \lla C^*(X) \lla C^*(X^e_{s-1}), $$
of $C^*(X^e_{s-1})$-modules, which we view as a cofibration of
$C^*(X^e)$-modules by pullback. Successive fibrations give
$$\Sigma_{n_i}C^*(X^e_i) \lla C^*(X^e_i) \lla C^*(X^e_{i-1}), $$
until we reach
$$\Sigma_{n_1}C^*(X^e_1) \lla C^*(X^e_1) \lla C^*(X^e), $$
so that
$$C^*(X)=C^*(X^e_s)\finbuilds C^*(X^e_0)=C^*(X^e)$$
as required.\qqed

\newcommand{\ext}{\mathrm{Ext}}     
\newcommand{\End}{\mathrm{End}}     

\section{Hochschild cohomology and  {pure Sullivan algebras}} 
\label{sec:bciiszci}
In this section we calculate the Hochschild cohomology of a  {pure } sci space
$X$ and upgrade the conclusion of Section \ref{sec:sciisbci} to give
the required conclusion that any  {pure } sci space is also zci.

\begin{thm}
\label{thm:sciiszci}
If $X$ is a  {pure } sci space of codimension $c$, then it is zci of codimension $c$.
\end{thm}

In view of Theorem \ref{thm:sciisbci} and Remark \ref{rem:bimodisallmod}
we need only show that the maps of bimodules used in the constructions of
Section \ref{sec:sciisbci} all lift to elements of $Z\Der (R)$.
Specifically, this is Corollary \ref{cor:lifting}.

\subsection{Hochschild cohomology.}
\label{subsec:Hochschild}
It is convenient to adapt the algebraic notation for Hochschild cohomology
to cochain algebras.

\begin{notation}
If $X \to Y$ is a fibration of spaces, set
$$HH^*(X|Y):=\ext^*_{\chains^*(X \times_Y X)}(\chains^*(X),\chains^*(X)).$$
 Note that we consider $\chains^*(X)$ as a $\chains^*(X \times_Y X)$-module via the diagonal map $X \to X \times_Y X$.
\end{notation}

We will be applying this to sci spaces, and use the notation of
Section \ref{sec:sciisbci}:
$X=X_s$ is an sci space of codimension $s$, so that for $1\leq i \leq s$
we have  fibrations
 $S^{n_i} \to X_i \to X_{i-1}$ with $n_i$ odd and $X_0=KV$, where $V$ is
finite dimensional and even.

\begin{thm}
\label{thm:Hochschildcalculation}
If $X$ is a  {pure } sci space as above, the Hochschild cohomology is given by
$$HH^*(X|KV)=H^*(X)[[\zeta_1,...,\zeta_s]], $$
where the degree of $\zeta_i$ is $n_i-1$.
\end{thm}

\begin{remark}
(i) The completion involved in forming the power series ring is
homogeneous, so that if $X$ is finite dimensional the ring is a polynomial
ring. Otherwise the formula is to be interpreted as formed by successive
adjunction of the variables in the stated order.

(ii) One would expect it to follow from Theorem
\ref{thm:Hochschildcalculation}
that the construction of
Section \ref{sec:sciisbci} can be upgraded to show $X$ is zci.
In any case, this upgrading of the construction is an ingredient in the
Hochschild cohomology calculation.
\end{remark}

The theorem evidently follows by repeated application of the following
general result about fibrations with fibre an odd sphere.

\begin{prop}
\label{prop:Hochschild}
Suppose given a fibration sequence $S^{2n+1}\lra Y \lra Z$, and a fibration $X \lra Y$.  Assuming (i) $\pi_*(Y) \lra \pi_*(Z)$ is surjective,
(ii) $\pi_*(X) \lra \pi_*(Y)$ is surjective, with kernel concentrated
in odd degrees,
and (iii) $\pi_*(X) $ is finite dimensional  {and $X$ is a pure space},
 we have
\[HH^*(X|Z)\cong HH^*(X|Y)[[\zeta]],\]
where $\zeta$ is of degree $2n$.
\end{prop}

This will be proved in Subsection \ref{subsec:Hochschildproof} below.

\subsection{Upgrading bimodule maps.}
Using the notation from Proposition \ref{prop:Hochschild},
we suppose given a fibration $S^{2n+1}\lra Y \lra Z$,
and a map $X \lra Y$. We write $A=\chains^*(X)$, $B=\chains^*(Y)$ and
$C=\chains^*(Z)$.

It is shown in Section \ref{sec:sciisbci}  (see Lemma \ref{lem:fib2} and
Proposition \ref{prop:loopspherebuild} ) that there is a cofibre
sequence of $A\otimes_C A$-modules
\[ A\otimes_C A \to A\otimes_B A \xrightarrow{\varphi}
\Sigma_{2n} A\otimes_B A.\]
To obtain an element of $Z\Der (R)$ we proceed as follows. For each
$A$-module $M$, we apply $-\otimes_A M$ to
obtain the cofibre sequence
\[ A\otimes_C M \to A\otimes_B M \xrightarrow{\varphi \otimes_A M} \Sigma_{2n}
A\otimes_B M.\]
To establish the zci condition, we must check it is natural for maps
of $A$-modules.

\begin{prop}
\label{prop:lifttoA}
There is a morphism
$\zeta:A \to \Sigma_{2n} A$ of $A\otimes_C A$-modules (i.e., an element
 $\zeta \in HH^{2n}(A|C)$) such that
\[ \zeta \otimes_B A \simeq \varphi.\]
\end{prop}

We will prove the proposition in Subsection \ref{subsec:lifting}
below. For the present we just observe that it has the desired consequence.

\begin{cor}
\label{cor:lifting}
There is a natural transformation $z$ of the identity functor on $A$-modules
such that for every $A$-module $M$
\[z(A \otimes_B M) \simeq \varphi \otimes_A M.\]
\end{cor}

\begin{proof}
The natural transformation $z$ that $\zeta$ induces on $A$-modules is given
by $z(M)=\zeta \otimes_A M$. We easily verify this has the
required property:
\begin{align*}
z(A \otimes_B M) &= \zeta \otimes_A (A \otimes_B M) \\
& = \zeta \otimes_B M \\
& = (\zeta \otimes_B A) \otimes_A M \\
& = \varphi \otimes_B M.
\end{align*}
\end{proof}

This completes the proof of  Theorem \ref{thm:sciiszci}.

\subsection{Models for spaces.}
Using the notation of Proposition \ref{prop:Hochschild}, we work with a
fibration sequence
$S^{2n+1} \to Y \to Z$ and a fibration $X \lra Y$, and we
let $A$, $B$ and $C$ be minimal Sullivan models for $X$, $Y$ and $Z$
 respectively. More explicitly, we take models as follows.
\begin{enumerate}
    \item $C=(\Lambda W ,d)$ with $W $ finite dimensional.
    \item $B=(\Lambda(W \oplus \Q x_{2n+1}),d)$,  containing
$C$ as a sub-algebra. We denote $W  \oplus \Q x$ by $V$.
    \item $A=(\Lambda(V \oplus U),d)$, containing $B$ as a sub-algebra,
    with  $U$ concentrated in odd degrees and finite dimensional.
     {We also assume that $d(U)\subset \Lambda W$, this is possible
    because $X$ has a pure Sullivan model.}
\end{enumerate}
Let $X^e_Y=X \times_Y X$ and let $X^e_Z=X \times_Z X$. The cochain algebras $A^e_B=A\otimes_B A$ and $A^e_C=A \otimes_C A$ are minimal Sullivan models for $X^e_Y$ and $X^e_Z$. We can write these cochain algebras explicitly as well:
\begin{itemize}
    \item $A^e_B=(\Lambda(V\oplus U_l \oplus U_r),d)$ where $U_l=\{u_l|u\in U\}$ and $U_r=\{u_r|u\in U\}$. The differential $d$ is the obvious one satisfying 
         {$d(u_l)=d(u_r)=d(u)\in\Lambda W$ and so $u_l-u_r$ is always a cocycle}.
    \item $A^e_C=(\Lambda(W \oplus \Q\{x_l,x_r\} \oplus U_l\oplus U_r),d)$.
\end{itemize}

\begin{remark}
There are three important morphisms for $X^e_Y$:
\begin{enumerate}
\item $l:X^e_Y \to X$ which is mapping to the left component,
\item $r:X^e_Y \to X$ which is mapping to the right component and
\item $\Delta:X \to X^e_Y$ which is the diagonal.
\end{enumerate}
The algebraic counterparts of these maps are
 $l: A \to A^e_B$, $r:A \to A^e_B$ and $\Delta: A^e_B \to A$.
The morphism $l$ is the map to the left component of $A \otimes_B A=A^e_B$,
it is defined by $l(u)=u_l$ for $u\in U$ and $l(v)=v$ for $v\in V$.
The description of $r:A \to A^e_B$ is precisely similar.
The diagonal map $\Delta:A^e_B \to A$ is defined by
 {$\Delta(u_l)=\Delta(u_r)=u$ and $\Delta$ is the identity on $\Q x$ }
and $W $.
There are similar maps for $X^e_Z$ and $A^e_C$. Note that $A$ is a
$A^e_C$-module via the morphism $\Delta : A^e_C \lra A$.
\end{remark}

\subsection{Some useful fibrations.}
The following fibrations will be central to the proof.
\begin{lemma}
\label{lem: magical fibrations}
There are the following fibration sequences,
\begin{enumerate}
    \item $X^e_Y \stackrel{\Delta}\to X^e_Z \to S(\Q x)$.
    \item $X \stackrel{\Delta}\to X^e_Y \to S(U)$.
    \item $X \stackrel{\Delta}\to X^e_Z \to S(U\oplus \Q x)$.
\end{enumerate}
where the bases are products of odd spheres with the indicated homotopy
groups.
\end{lemma}
\begin{proof}
We reformulate the lemma algebraically: there are the following cofibration sequences of Sullivan algebras:
\begin{enumerate}
    \item $L_1=(\Lambda \Q x',0) \to A^e_C \xrightarrow{\Delta} A^e_B$,
 where $x'\mapsto x_l-x_r$.
    \item $L_2=(\Lambda U',0) \to A^e_B \xrightarrow{\Delta} A$,
 where $U' \cong U$ and $u' \mapsto u_l-u_r$.
    \item $L_1\otimes_{\Q}L_2=(\Lambda (\Q x' \oplus U'),0) \to A^e_C \xrightarrow{\Delta} A$.
\end{enumerate}
Since the composite $(\Lambda \Q x',0) \to A^e_C \xrightarrow{\Delta} A^e_B$ is the trivial morphism, there is a natural morphism $\epsilon:A^e_C \otimes_{L_1} \Q \to A^e_B$. It is easy to see that $\epsilon$ is an isomorphism on homotopy groups. The proof for the two other cofibration sequences is similar.
\end{proof}

We shall make two uses of these fibrations. The first use is to build relative
cofibrant models for our cochain algebras. An $A^e_C$-cofibrant model for
$A^e_B$ is
\begin{center}
$\tilde{A^e_B}=(\Lambda(W  \oplus \Q\{x_l,x_r,z_{2n}\} \oplus U_l\oplus U_r), d)$ where $dz=x_l-x_r$.
\end{center}
It is easy to see that the obvious morphism $\tilde{A^e_B} \to A^e_B$ is indeed a weak equivalence. Similarly a cofibrant model for $A$ over $\tilde{A^e_B}$ (and therefore also over $A^e_C$) is given by the formula
\begin{center}
$\tilde{A}=(\Lambda(W  \oplus \Q\{x_l,x_r,z_{2n}\} \oplus U_l\oplus U_r\oplus \tilde{U}), d)$ where $\tilde{U}=\Sigma_{1}U$ and $d(\tilde{u})=u_l-u_r$.
\end{center}

The second use of Lemma~\ref{lem: magical fibrations} is in defining a strange and useful space $T$.
\begin{lemma}
\label{lem:magicalT}
Let $T$ be the homotopy fibre of the map $X^e_Z \to S(U')$.
\begin{enumerate}
\item There is a fibration sequence $X \to T \to S(\Q x)$.
\item The following is a homotopy pullback square
\[\xymatrix{ X \ar[r]^{\Delta} \ar[d] & X^e_Y \ar[d]\\ T \ar[r] & X^e_Z}.\]
\end{enumerate}
\end{lemma}
\begin{proof}
An $A^e_C$-cofibrant cochain model for $T$ is given by the formula
\begin{center}
$F=(\Lambda(W \oplus Q\{x_l,x_r\} \oplus U_l\oplus U_r\oplus \tilde{U}), d)$
where $\tilde{U}=\Sigma_{1}U$ and $d(\tilde{u})=u_l-u_r$.
\end{center}
 {Note that $u_l-u_r$ is a cocycle because $d(U)\subset \Lambda W$. }
From this model the fibration sequence $X \to T \to S(\Q x)$ is evident.

Let $X'$ be the homotopy pullback of the diagram:
\[ \xymatrix{  & X^e_Y \ar[d]\\ T \ar[r] & X^e_Z.}\]
A cochain algebra model for $X'$ is $F \otimes_{A^e_C} \tilde{A^e_B} = (\Lambda(W \oplus Q\{x_l,x_r\} \oplus U_l\oplus U_r\oplus \tilde{U}\oplus\Q z_{2n}), d)$ which is clearly isomorphic to $\tilde{A}$. Moreover, the morphism $\tilde{A^e_B} \to F \otimes_{A^e_C} \tilde{A^e_B} \cong \tilde{A}$ is indeed the diagonal $\Delta$.
\end{proof}

\subsection{Lifting the map of bimodules.}
\label{subsec:lifting}
We can now prove Proposition \ref{prop:lifttoA}. Recall the cochain model $F$
for $T$ given in the proof of Lemma \ref{lem:magicalT}.
The fibration $X \to T \to S(\Q x)$ induces an exact sequence of $F$-modules
(and therefore also of $A^e_C$-modules):
\[ F \to \tilde{A} \xrightarrow{\zeta} \Sigma_{2n} \tilde{A}.\]
We will require an explicit description of $\zeta$. It is defined by
\begin{itemize}
\item $\zeta(fz^q)=fz^{q-1}$ and
\item $\zeta(f)=0$ if $f$ is not divisible by $z$.
\end{itemize}

Similarly, the fibration sequence $X^e_Y \to X^e_Z \to S(\Q x)$ gives rise to an exact sequence
\[ A^e_C \to \tilde{A^e_B} \xrightarrow{\varphi} \Sigma_{2n} \tilde{A^e_B}\]
of $A^e_C$-modules. An explicit description of $\varphi$ is given by
\begin{itemize}
\item $\varphi(fz^q)=fz^{q-1}$ and
\item $\varphi(f)=0$ if $f$ is not divisible by $z$.
\end{itemize}

As a cofibrant replacement of $B$ over $B \otimes_C B$, we take
$\tilde{B}=(\Lambda(W \oplus \Q\{x_l,x_r,z_{2n}\}),d)$ with $dz=x_l-x_r$. We can now prove the following proposition, which is an explicit cochain level
version of Proposition \ref{prop:lifttoA}.
\begin{prop} There is a natural equivalence
\[  \varphi \simeq  \zeta \otimes_{\tilde{B}} \tilde{A}\]
\end{prop}
\begin{proof}
We shall define a cochain algebra model $\hat{A}$ for $X$, which is cofibrant over $\tilde{B}$. Let $\hat{A}=(\Lambda(W  \oplus \Q\{x_l,x_r,z_{2n}\} \oplus U_l), d)$ where $dz=x_l-x_r$. The morphism $\zeta$ is equivalent to the obvious morphism $\hat{\zeta}:\hat{A} \to \Sigma_{2n} \hat{A}$. It is now easy to see there is an equality of morphisms of $A^e_C$-modules:
\[ \varphi = \hat{\zeta} \otimes_{\tilde{B}} \hat{A}\]
\end{proof}

\begin{remark}
Proposition \ref{prop:loopspherebuild}
shows that the fibration $S^{2n+1} \to Y \to Z$ yields an exact sequence:
\[B\otimes_C B \to \tilde{B} \xrightarrow{\psi} \Sigma_{2n} \tilde{B}.\]
It is easy to see that $\varphi = A \otimes_B \psi \otimes_B A$ (note that $\psi$ is a morphism of $B\otimes_C B$-modules, which justifies tensoring over $B$ on the left and right).
\end{remark}
\subsection{Proof of Proposition \ref{prop:Hochschild}}
\label{subsec:Hochschildproof}

Using our cofibrant models, we have explicit complexes
for calculating Hochschild cohomology:
$$HH^*(A|C)=H^*(\End_{A^e_C}(\tilde{A}))
\mbox{ and } HH^*(A|B)=H^*(\End_{\tilde{A^e_B}}(\tilde{A})).$$
In these terms, we may state Proposition \ref{prop:Hochschild}
more explicitly as follows.

\begin{prop}
\[ H_*(\End_{A^e_C}(\tilde{A})) = H_*(\End_{\tilde{A^e_B}}(\tilde{A}))[[\zeta]] \]
\end{prop}
\begin{proof}
Let $R=H^*(\End_{A^e_C}(\tilde{A}))$ and let $Q=
H^*(\End_{\tilde{A^e_B}}(\tilde{A}))$.
Note that both $R$ and $Q$ are graded-commutative, because Hochschild cohomology is always graded-commutative.
To prove the proposition we need several ingredients. The first ingredient is a short exact sequence
\[ \Sigma^{2n} R \xrightarrow{\zeta} R \to Q\]
of $R$-modules.

Consider the morphism $F \xrightarrow{p} \tilde{A}$ of
Lemma \ref{lem:magicalT}.
Applying the functor $\Hom_{A^e_C}(-,\tilde{A})$ yields a morphism
$\End_{A^e_C}(\tilde{A}) \xrightarrow{p^*} \Hom_{A^e_C}(F,\tilde{A})$.
Since $\tilde{A}$ is a $\tilde{A^e_B}$-module there is an adjunction:
\[ \Hom_{A^e_C}(F,\tilde{A}) \cong \Hom_{\tilde{A^e_B}}(\tilde{A^e_B}
\otimes_{A^e_C} F,\tilde{A}) \cong \End_{\tilde{A^e_B}}(\tilde{A})\]
(the isomorphism $\tilde{A^e_B}\otimes_{A^e_C} F \cong \tilde{A}$ is Part 2
of Lemma~\ref{lem:magicalT}).
Thus $p^*$ is a map $\End_{A^e_C}(\tilde{A}) \to \End_{\tilde{A^e_B}}(\tilde{A})$.
On the other hand, we have the natural multiplicative change of rings map
$\iota: \End_{\tilde{A^e_B}}(\tilde{A}) \to \End_{A^e_C}(\tilde{A})$.
Using the explicit construction of internal Hom of DG-modules over a
CDGA, it is straightforward to verify that $p^*$ is left inverse to $\iota$.

Next, consider the short exact sequence $ F \to \tilde{A} \xrightarrow{\zeta} \Sigma_{2n} \tilde{A}$. Applying $\Hom_{A^e_C}(-,\tilde{A})$ to this sequence yields a distinguished triangle of left $\End_{A^e_C}(\tilde{A})$-modules
\[ \Sigma^{2n} \End_{A^e_C}(\tilde{A}) \xrightarrow{\zeta^*} \End_{A^e_C}(\tilde{A}) \xrightarrow{p^*} \End_{\tilde{A^e_B}}(\tilde{A}).\]
The morphism $\zeta^*:\Sigma^{2n} \End_{A^e_C}(\tilde{A}) \to \End_{A^e_C}(\tilde{A})$ is just composition with $\zeta$, i.e., right multiplication by $\zeta \in \End_{A^e_C}(\tilde{A})$. This distinguished triangle yields a long exact sequence of homology groups. Since $H_*(p^*)$ is an epimorphism,
 we have  a short exact sequence of graded left $R$-modules:
\[ \Sigma^{2n} R \xrightarrow{\zeta} R \xrightarrow{H_*(p^*)} Q.\]

Note that there are two multiplicative structure on $Q$: the usual one and
the one coming from $Q$ being a quotient of the graded ring $R$ by the ideal $(\zeta)$. These structures must coincide, because $H_*(\iota)$ is multiplicative and has a left inverse.

The second ingredient is that the homotopy inverse limit  of the tower
\[ \mathcal{T}:=\left[ R\xleftarrow{z} R \xleftarrow{z} \cdots \right] \]
is zero.
Because $A$ is zero in negative codegrees, the homotopy colimit
$\tilde{A}^\infty$ of
the telescope $\tilde{A} \xrightarrow{\zeta} \Sigma_{2n} \tilde{A} \xrightarrow{\zeta} \cdots$ is zero.
Applying $\Hom_{A^e_C}(-,\tilde{A})$ to this telescope gives a tower
\[ \End_{A^e_C}(\tilde{A}) \xleftarrow{\zeta^*} \Sigma^{2n}\End_{A^e_C}(\tilde{A}) \xleftarrow{\zeta^*} \cdots\]
Its homotopy inverse limit, $\Hom_{A^e_C}(\tilde{A}^\infty,\tilde{A})$
is therefore also zero. By the Milnor exact sequence,
$\ilim \mathcal{T}=0$ and $R^1 \ilim \mathcal{T}=0$.

Next, consider the two towers:
\begin{enumerate}
\item $\mathcal{U}=\left[ R \xleftarrow{=} R \xleftarrow{=} \cdots \right],$
 and
\item $\mathcal{V}=\left[
Q=R/(\zeta) \leftarrow R/(\zeta^2) \leftarrow R/(\zeta^3) \leftarrow \cdots
\right]$.
\end{enumerate}
There is a short exact sequence of towers $0 \to \mathcal{T} \to \mathcal{U} \to \mathcal{V} \to 0$, given by:
\[ \xymatrix{
R \ar[d]^\zeta & R \ar[l]_{\zeta} \ar[d]^{\zeta^2} & R \ar[l]_\zeta \ar[d]^{\zeta^3} & \ar[l]_\zeta {\cdots} \\
R \ar[d] & R \ar[d] \ar[l]_{=} & R \ar[d] \ar[l]_{=} & \ar[l]_{=} {\cdots} \\
Q & {R/(\zeta^2)} \ar[l] & {R/(\zeta^3)} \ar[l] & \ar[l] {\cdots} } \]
The six term exact sequence:
\[ 0 \to \ilim \mathcal{T} \to \ilim \mathcal{U} \to \ilim \mathcal{V}
\to R^1\ilim \mathcal{T} \to R^1\ilim \mathcal{U} \to R^1\ilim \mathcal{V} \to 0\]
shows that $R$ (which is isomorphic to $\ilim \mathcal{U}$) is isomorphic to $\ilim \mathcal{V}$.

To complete the proof we need to show that $R/(\zeta^n)$ is isomorphic to the truncated
polynomial ring $Q[\zeta]/(\zeta^n)$. Observe that $(\zeta^n)/(\zeta^{n+1})$ is
isomorphic to $Q$ as an $R$-module, and that there is a subalgebra
$Q[\zeta] \subseteq R$. These facts yield a morphism of short exact sequences:
\[ \xymatrix{
Q\ar[d]^{\cong} \ar[r] & {Q[\zeta]/(\zeta^{n+1})} \ar[r] \ar[d] & {Q[\zeta]/(\zeta^{n})} \ar[d] \\
{(\zeta^n)/(\zeta^{n+1})} \ar[r] & {R/(\zeta^{n+1})} \ar[r] & {R/(\zeta^{n})} }\]
Since $R/(\zeta)\cong Q$ we get an inductive argument showing that $R/(\zeta^n)\cong Q[\zeta]/(\zeta^n)$ as rings.

Therefore
\[ R \cong \ilim Q[\zeta]/(\zeta^n) = Q[[\zeta]],\]
as required.
\end{proof}

\section{Polynomial growth implies spherical extension.}
\label{sec:gciissci}

The purpose of the present section is to complete the loop of implications
and prove the following theorem.

\begin{thm}
\label{thm:gciissci}
If $X$ is a gci space, it is also sci.
\end{thm}

This states that a finiteness condition (Noetherian cohomology
and finite homotopy) implies that a space has a particular form
(fibration $F \lra X \lra KV$, where $\pi_*(F)$ is in odd degrees)
and is therefore perhaps the most interesting step.

\subsection{Strategy.}
\label{subsec:gciisscistrategy}
Assume $X$ is gci. By the Milnor-Moore theorem,
$\pi_*(\Omega X)$ is finite dimensional, and in
particular its even part is finite dimensional.

We argue by induction on $\dimQ (\pi_*(X))$. The result is
trivial if $\pi_*(X)=0$, and the inductive step will be to remove
the top homotopy group and retain the Noetherian condition. Suppose
then that the top non-zero homotopy is in degree $s$ and that
$0\neq x\in \pi_s^{\vee}(X)$.

If $s=2n-1$ is odd, killing homotopy groups gives a fibration
$$S^{2n-1} \lra X \lra X',  $$
with $\dimQ(\pi_*(X'))=\dimQ (\pi_*(X))-1$.
In the rational setting, the fibration is principal, so there
is also a fibration
$$X \lra X'\lra K(\Q, 2n-2),  $$
and we may use the Serre spectral sequence to deduce that
$H^*(X')$ is Noetherian. Thus $X'$ is gci, and by induction we
conclude it is also sci. The fibration displays $X$ as being
sci as required.

If $s=2n$ is even, killing homotopy groups gives a fibration,
$K(\Q, 2n) \lra X \lra Y$, but this is not of use to us.
We will argue, heavily using the fact that $H^*(X)$ is Noetherian,
that in fact the element $x$ is in the image of the dual Hurewicz map.
Accordingly there is another fibration
$$X' \lra X \lra K(\Q, 2n)$$
where $\dimQ (\pi_*(X'))= \dimQ (\pi_*(X))-1 $.
Applying the Serre spectral sequence to the fibration
$$S^{2n-1} \lra X'\lra X,$$
we see that $H^*(X')$ is Noetherian. By induction we conclude $X'$ is sci,
and from Corollary \ref{cor:scirearrangement} it follows that $X$ is sci.

\subsection{The dual Hurewicz map.}
In rational homotopy it is natural to dualize the Hurewicz map
$$h:\pi_n(X) \lra H_n(X)$$
and concentrate on the dual Hurewicz map
\[ h^{\vee}:H^*(X) =H_*(X)^{\vee}\to \pi_*(X)^\vee.\]
In fact, the dual Hurewicz map $h^{\vee}$ must be zero on decomposable elements
of $H^*X$, and so it yields a map from the indecomposable quotient of
$H^*(X)$ to $\pi_*(X)^\vee$.

If $(\Lambda V,d)$ is a minimal Sullivan model for $X$ then this dual
Hurewicz map is  the linear map
\[ h^{\vee}: H^*(\Lambda V,d) \to V\]
that  comes from dividing $(\Lambda V,d)$ by the sub-cochain complex
$(\Lambda^{\geq 2} V,d)$.
An element $x \in V$ is in the image of $h^{\vee}$ if and only if
there is a $g\in \Lambda^{\geq 2} V$ such that $d(x+g)=0$.

\subsection{The dual Hurewicz map and the Noetherian condition.}

It is immediate from Theorem \ref{thm:sci} that if $X$ is sci then $h^{\vee}$
is an epimorphism in even degrees. The critical step in showing that gci
implies sci is to prove a special case of this surjectivity
holds for gci spaces. We are grateful to S.Iyengar for pointing out that
a corresponding result with a very different proof appears as a crucial lemma
in \cite{DVP}.

\begin{prop}
\label{prop:NoethHurepi}
Suppose $X$ is a rational space with finite dimensional homotopy
and that the  top degree in which homotopy is nonzero is  $2n$. If $H^*(X)$ is
Noetherian then the dual Hurewicz map
$$h^{\vee}: H^{2n}(X)\lra \pi_{2n}^{\vee}(X)$$
is surjective.
\end{prop}

\begin{proof}
 By killing homotopy groups, there is a fibration sequence
\[ Y \leftarrow X \stackrel{\Phi}\leftarrow K(\Q,2n)\]
so that $\Phi$ a monomorphism in homotopy.
We suppose $S=(\Lambda W,d)$ is a minimal Sullivan model for $Y$,
and  $R=(\Lambda V,d)$ is a minimal Sullivan model for $X$, where
$V=W\oplus\Q x$. We can take $Q=(\Lambda x,0)$ as a Sullivan model
for $K(\Q,2n)$ with $\Phi:R \to Q$ being the obvious map, so that
\[ S \rightarrow R \stackrel{\Phi}\rightarrow Q \simeq R\otimes_S \Q .\]
provides an algebraic model of the fibration.

To conclude, we apply the following lemma.


\begin{lemma}
\label{lem: Zero on homology map}
If $H^*(\Phi)$  is non-trivial
then $x$ is in the image of the dual Hurewicz map.
\end{lemma}
\begin{proof}
Suppose
$$x^n\in \mathrm{im} (H^*(X)\rightarrow H^*(K(\Q,n))=\Q[x]).$$
This implies there  is a cocycle in $R$ of the form:
\[ x^n + f_1 x^{n-1} + f_2 x^{n-2} +\cdots + f_n \]
with $f_i \in \Lambda W$. The differential of this cocycle is
\[ (ndx+df_1)x^{n-1} + \mbox{ [ terms of lower degree in $x$ ]} =0,\]
which implies $ndx+df_1=0$. Hence there is an element $g\in R$ such that
$x+g$ is a cocycle in $R$. The element $x+g$ cannot be a coboundary
because $R$ is  minimal, and hence $x$ is in the image of the dual Hurewicz map.
\end{proof}

If $H^*(X)$ is Noetherian,  then by \cite[9.3]{DGI1},
the stable Koszul complex can be used to construct the $\Q$-cellularization
and  there is a local cohomology spectral sequence
\[ H^{-p}_I(H^*(Q))_q \ \Rightarrow \ H_{p+q}(\cell^R_{\Q} (Q)), \]
where $Q$ is considered as an $R$-module via $\Phi$.

We now suppose $x$ is not in the image of the dual Hurewicz map and
 deduce two contradictory statements about $\cell^Q_{\Q}(Q)$. First,
Lemma~\ref{lem: Zero on homology map} implies that $H^*(\Phi)$ is trivial and
hence the spectral sequence collapses at the $E^2$-page to show
\[ \cell_{\Q}^R(Q)\simeq Q. \]
In particular $\cell_{\Q}^R(Q)$ has cohomology only in codegrees
$\geq 0$.

On the other hand,  if we assume $x$ is not in the image of
the dual Hurewicz map,  we will see that the cohomology of $\cell_{\Q}^R(Q)$
must be quite different.

\begin{lemma}
\label{lem:splitsthrough}
If $K(\Q, 2n) \lra X \lra Y$ is a fibration sequence killing the top
homotopy group of $X$ then $C^*(\Omega Y)$ is $\Q$-cellular as
a module over $C^*(K(\Q, 2n))$.


\end{lemma}

\begin{proof}
Since the homotopy groups of $\Omega Y $ are concentrated in degrees less
than $2n-1$ and $\Omega Y$ is a product of Eilenberg-MacLane spaces,
the connecting map  $\Omega Y \lra K(\Q, 2n)$ is null. It follows that
 the induced map in cohomology
$\Q [x]=H^*(K(\Q, 2n))\lra H^*(\Omega Y)$ factors through $\Q$, and hence
that $C^*(\Omega Y)$ is $\Q$-cellular as required.
%
\end{proof}

The third author has identified two elementary but very useful
base change results for cellularization.
\begin{lemma}{\bf (Independence of base \cite[3.1]{Shamir})}
Suppose $R\lra S$ is a map of rings.

(i) {\bf (Strong)}
If $B$ is an $S$-module and $S\tensor_RB$ is $B$-cellular over $S$,
then for any $S$-module $Y$
$$\cell_{B}^RY \simeq \cell_{B}^SY. $$

(ii){\bf (Weak)}
If $A$ is an $R$-module and $S\tensor_RA$ is $A$-cellular over $R$,
then for any $S$-module $Y$
$$\cell_{A}^RY \simeq \cell_{S\tensor_R A}^SY. \qqed$$
\end{lemma}

It follows from Lemma \ref{lem:splitsthrough}
that $Q\otimes_R \Q$ is $\Q$-cellular as an $Q$-module.
We may therefore apply  the Strong Independence of Base property
to the map $R \to Q$ with $B=\Q$ and conclude
\[ \cell_{\Q}^R (Q) \simeq \cell_{\Q}^Q (Q).\]
Since $H^*(Q)=\Q [x]$, we can easily compute
$$H^*(\cell_{\Q}^Q (Q))=\Sigma H^1_{I}(\Q[x])=\Sigma \Q [x]^{\vee}$$
using the spectral sequence above.
In particular $\cell_\Q^Q Q$ has cohomology in negative
codegrees. It is therefore not equivalent to $Q$, and the assumption
that $x$ is not in the image of the dual of the Hurewicz map leads to a
contradiction.

This completes the proof of the proposition.
\end{proof}

This completes the proof of Theorem \ref{thm:gciissci}.
\qqed

\section{Examples.}
\label{sec:examples}
It is quite easy to construct examples in rational homotopy theory,
so we can see that various classes are distinct.

\subsection{Homotopy invariant notions and cohomology
rings.}

We may impose a homotopy invariant condition on a space $X$ or
a conventional condition on the cohomology ring $H^*(X)$.

In the regular case there is no distinction by
Theorem \ref{thm:hregisKV},  since
rational graded connected commutative rings are regular if and only
if they are polynomial on even degree generators.

In the ci case the homotopy invariant notion is strictly
weaker than the notion for cohomology rings.
Indeed, we show in Proposition \ref{prop:ciishci}
that if $H^*(X)$ is ci then $X$ is sci. On the other hand,
 Example \ref{eg:hGornotGor} gives a  {pure } sci
space whose cohomology ring is not even Gorenstein.

\begin{prop}
\label{prop:ciishci}
If $H^*(X)$ is a complete intersection, then $X$ is formal, and
there is a fibration
$$S^{m_1}\times \cdots \times S^{m_c} \lra X \lra KV$$
with $m_1,m_2, \ldots , m_c$ odd. In particular, $X$ is also sci.
\end{prop}

\begin{remark}
By Theorem \ref{thm:sci} a general sci space $X$ has a similar fibration
with fibre an arbitrary space with finite dimensional odd homotopy.
Those with $H^*(X)$ ci also have zero Postnikov invariants:
there are vastly more sci spaces than those with ci cohomology.
However Example \ref{eg:hGornotGor} shows that even when the fibre is a product
of odd spheres, the cohomology ring need not be ci.
\end{remark}

\begin{proof}
We may suppose $H^*(X)=k[x_1, \ldots, f_n]/(f_1,\ldots , f_c)$ for
suitable even degree generators $x_1, \ldots , x_n$ and regular sequence
$f_1, \ldots , f_c$.  Now let $V$ be a graded vector space with
basis $x_1', \ldots , x_n'$, where the degree of $x_i'$ is the same as
that of $x_i$, and let $W$ be a graded vector space
with basis $\phi_1', \ldots , \phi_c'$ where the codegree of $\phi_i'$
is one less than that of $f_i$. We now take $X'$ to have
model $M(X')=(\Lambda (V)\sdr \Lambda (W), d\phi_1'=f_1',
\ldots , d\phi_c'=f_c')$.
Since $f_1, \ldots , f_c$ is a regular sequence,  $H^*(X')\cong H^*(X)$.

Now construct a map
$$g: M(X') \lra C^*(X)$$
by taking $g(x_i')$ to be a representative
cycle for $x_i \in H^*(X)$. Since $f_i$ is trivial in
$H^*(X)$, we may choose $\phi_i\in C^*(X)$ so that $d\phi_i =f_i$, and
define $g(\phi_i')=\phi_i$. The resulting map $g$ is a cohomology
isomorphism and therefore an equivalence.

The structure of $M(X')$ gives a fibration as claimed.
\end{proof}

Similarly, in the Gorenstein case the homotopy invariant notion is strictly
weaker. On the one hand,  Corollary \ref{cor:GorishGor} shows that if
$H^*(X)$ is Gorenstein, then $X$ is h-Gorenstein. If $X$ is h-Gorenstein
and $H^*(X)$ is Cohen-Macaulay then $H^*(X)$ is Gorenstein. However
 Example \ref{eg:hGornotGor} gives an h-Gorenstein
space whose cohomology ring is not Cohen-Macaulay.

\subsection{Separating the hierarchy.}

Since h-regular spaces are of the form $KV$, it is easy
to see that there are sci spaces that are not regular.
To give an example of an h-Gorenstein space that is not
gci, we may use connected sums of manifolds, as in Example
\ref{eg:hGornothci}. Finally, there are many spaces with
Noetherian cohomology that are not h-Gorenstein: two easy
sources of examples are either
finite dimensional spaces whose cohomology ring does not
satisfy Poincar\'e duality, or Cohen-Macaulay rings which are
not Gorenstein.

\begin{example}
\label{eg:hGornothci}
We provide a space $X$ with $H^*(X)$ so that $X$ is h-Gorenstein
but not gci. Almost any non-trivial connected sum of manifolds will do,
but we give an explicit example.

First, note that if $M$ and $N$ are manifolds, their connected sum
$$M\hash N =(M'\vee N')\cup e^n$$
where $M'$ is $M$ with a small disc removed, and similarly for $N'$.
By considering Lie models as in \cite[24.7]{FHT}, we obtain
$$\pi_*(\Omega (M\hash N))=(\pi_*(\Omega M') *
\pi_*(\Omega N'))/(\alpha+\beta), $$
where $*$ is the coproduct of graded Lie algebras, and where
$\alpha$ and $\beta$ are the attaching maps for the top cells in $M$ and
$N$.

Perhaps the simplest thing to try is $M=N=\CP^2$. Here we obtain
$$\pi_*(\Omega(\C P^2 \hash \C P^2))=
\mathrm{Lie}(u_1,v_1)/([u_1,u_1]+[v_1,v_1]),  $$
where $\mathrm{Lie}(V)$ denotes the free graded Lie algebra on $V$.
This shows $\pi_*(\Omega(\C P^2 \hash \C P^2))$ is finite, so
$\C P^2 \hash \C P^2$ is  gci, which also follows from
the fact that its cohomology ring
$$H^*(\C P^2 \hash \C P^2)=\Q [a_2,b_2]/(a^2,b^2)$$
is a complete intersection.

However, once  we take three copies, we obtain
$$\pi_*(\Omega( \C P^2 \hash \C P^2 \hash \C P^2))=\mathrm{Lie}(u_1,v_1,w_1)/
([u_1,u_1]+[v_1,v_1]+[w_1,w_1]), $$
which is not finite, so that
$\C P^2 \hash \C P^2 \hash \C P^2$ is not  gci, giving the
required example.
\end{example}

 {
\begin{example}
\label{eg:nonZCI}
Here is an example of a space which is gci but not zci. Let $X$ be the space
with model $R=(\Lambda (u_3,v_3,w_5), dw=uv)$. There is a map of graded rings
$\Psi: Z\Der (R) \to H_*(\Omega X)$ given by
\[  \Psi(\zeta) \ = \ (\zeta_\Q:\Q \to \Sigma_n \Q) \in H_n(\Omega X)\]
Clearly the image of $\Psi$ is contained in the center of $H_*(\Omega X)$. The
graded ring $H_*(\Omega X)$ is the enveloping algebra of the graded Lie algebra $L$,
where $L$ is generated by three elements $U_2$, $V_2$ and $W_4$ and a single
relation $UV=W$. The center of $H_*(\Omega X)$ is therefore the set $\{\Q W^n|n\geq0\}$.
Now suppose $X$ was zci, then we would have had appropriate elements
$\zeta_1,...,\zeta_n \in Z\Der(R)$. Since the degree of $\zeta_i$ is non zero,
$\Psi(\zeta_i)$ is either zero or $a_i W^{n_i}$ for some $n_i>0$.

Let $M$ be the $R$-module that is the cone of the map $\Q \xrightarrow{W} \Sigma_4\Q$.
The module $M/\zeta_1/\cdots/\zeta_n$ must be a small $R$-module. Now consider
the $C_*(\Omega X)$-module $\bar{M}=\Hom_{R}(\Q,M)$. By the Yoneda lemma, the map
\[\bar{\zeta}_i=\ext_R^*(\Q,\zeta_i):H_*(\bar{M}) \to H_{*+|\zeta_i|}(\bar{M})\]
of $H_*(\Omega X)$-modules is simply multiplication by $\Psi(\zeta_i)$. 
It is easy to see that
\[H_*(\bar{M})=H_*(\Omega X)/(W)=\Q[U,V]\]
and therefore the induced map $\bar{\zeta}_i$ is zero on the homology of $\bar{M}$.
We conclude that $\bar{M}/\bar{\zeta}_1/\cdots/\bar{\zeta}_n$ has infinitely many
nonzero homology groups. However,
\[ \bar{M}/\bar{\zeta}_1/\cdots/\bar{\zeta}_n \simeq \Hom_{R}(Q,M/\zeta_1/\cdots/\zeta_n)\]
and since $M/\zeta_1/\cdots.../\zeta_n$ is small and $C^*(X)$ is h-Gorenstein we see that
the $C_*(\Omega X)$-module $\Hom_{R}(Q,M/\zeta_1/\cdots/\zeta_n)$ has only finitely many
nonzero homology groups, in contradiction.
\end{example}
}

\subsection{Miscellaneous examples.}

The following example shows that the Noetherian condition is essential
in the definition of gci.

\begin{example}
\label{eg:nonNoetherian}
The space $X$ with model $(\Lambda (v_{2a},x_{2b+1},w_{2a+2b}), dw=vx)$
is not sci. For definiteness, we work with $a=b=1$, so
that we have the model $(\Lambda (v_2,x_3,w_4), dw=vx)$.

Indeed, if $X$ is sci it must be in a fibration
$S^3 \lra X \lra KV$
where $V=\Q\{ v,w\}$. But then in homotopy we have
a short exact sequence
$$0 \lra \pi_*(\Omega S^3)\lra \pi_*(\Omega X) \lra \pi_*(\Omega KV)\lra 0$$
of graded Lie algebras, which implies that $\pi_*(\Omega S^3)$
is an ideal of $\pi_*(\Omega X)$. On the other hand, since $dw=vx$,
the corresponding elements $\overline{v}_1, \overline{x}_2$ and $\overline{w}_3$
 in the Lie algebra $\pi_*(\Omega X)$ satisfy
$\overline{w}=[\overline{v}, \overline{x}]$, and we have a contradiction
since $\pi_*(\Omega S^3)$ is generated by $\overline{x}$.

This is consistent with our general results since the cohomology ring
is not Noetherian. Indeed, the cohomology ring $H^*(X)$ is $\Q [v]$
in even degrees, whilst all products of the odd degree elements
$x,wx,w^2x, \ldots$ are zero.
\end{example}

\section{The nci condition}
\label{sec:nci}

In this section we conclude by  giving a condition in the style of the
zci condition which captures polynomial growth in the non-Noetherian
situation. From another point of view, since Example \ref{eg:nonNoetherian}
shows that the Noetherian condition is essential, the nci condition
introduced here is genuinely weaker than both the zci  {and the eci } 
conditions.
The letter `n' in nci stands for nilpotent.

\subsection{The condition.}

We suppose that $R$ is a CDGA and continue to write
 $\Der (R)$ for the derived category of dg-$R$-modules.

\begin{defn}
We say that $R$ is \emph{nci of length $\leq n$} if there are
\begin{enumerate}
\item a sequence of triangulated subcategories $\Der (R)=\Der_0 \supseteq \Der_1 \supseteq \cdots \supseteq \Der_n$ (not necessarily full) and
\item natural transformations $\zeta_i:1_{\Der_i} \to \Sigma_{|\zeta_i|} 1_{\Der_i}$ for $i=0,...,n-1$
\end{enumerate}
such that the following conditions hold
\begin{enumerate}
\item Every $\Der_i$ is closed under coproducts.
\item Every $\zeta_i$ is central among natural transformations of $1_{\Der_i}$.
\item For every $X \in \Der_i$ there exists an object $X/\zeta_i \in \Der_{i+1}$ and a distinguished triangle $X \xrightarrow{\zeta_i} \Sigma_{|\zeta_i|} X \to X/\zeta_i$.
\item If $0 \neq X \in \Der_i$ is in the thick subcategory generated by $\Q$ then $X/\zeta_i$ is non-zero.
\item If $0 \neq X \in \Der (R)$ is in the thick subcategory generated by $\Q$ then $X/\zeta_0/\zeta_1/\cdots/\zeta_{n-1}$ is non-zero and small as an object of $\Der (R)$.
\end{enumerate}
\end{defn}

\begin{remark}
(i) Note first that if $R$ is zci  {or eci } of codimension $c$, it is clearly
nci of length $c$.

(ii) On the other hand, if $R$ is nci and the natural
transformations $\zeta_1,...,
\zeta_{n-1}$ can be extended to central natural transformations of
$1_{\Der (R)}$, then $R$ is almost zci (the only additional
condition required is that the cohomology be Noetherian).
Remark~\ref{rem: unextendable natural transformation} below provides
an explicit example where it is not possible to extend one of these natural transformations.
\end{remark}

In the context of rational homotopy theory there is a straightforward
characterization of nci CDGAs.

\begin{thm}
\label{the: Main theorem}
Let $R=(\Lambda V,d)$ be a minimal Sullivan algebra. Then $R$ is nci if and only if $V$ is finite dimensional.
\end{thm}

Lemma~\ref{len: Exponential growth implies not nci} below
shows  that if  $V$ is infinite dimensional, then  $R$ is not nci.
The converse is proved in Subsection \ref{sub: Proving the main theorem} below.

\begin{remark}
The contrast with  {eci } spaces, where Theorem \ref{thm:sci} shows the
structure
is much more constrained (the differential on even generators is zero) is
very striking.

On the other hand, the only difference is the Noetherian condition.
Indeed, if $X$ is nci and $H^*(X)$ is Noetherian, then
Theorem \ref{the: Main theorem} shows that $X$ is  gci
and therefore (by Theorems \ref{thm:gciissci} and \ref{thm:sciiszci}) also
 {eci}.
\end{remark}

\subsection{Exponential Growth}

The hard work in this section  is in dealing  with the case of a natural
transformation of degree zero. This may be a useful counterpart to the
approach to the Jacobson radical in \cite{BGzci}.

\begin{lemma}
\label{len: Exponential growth implies not nci}
Let $X$ be a simply-connected finite CW-complex. If $C^*(X)$ is nci,
then  $H_* (\Omega X)$ has polynomial growth.
\end{lemma}
\begin{proof}
Since $R=C^*(X)$ is nci there is a sequence of $R$-modules
$M_0,M_1,...,M_n$ such that
\begin{enumerate}
\item $M_0=\Q$ and $M_i \in \Der_i$.
\item There is a distinguished triangle $\Sigma^{|\zeta_{i}|} M_i \xrightarrow{\zeta_{i}} M_i \to M_{i+1}$.
\item $M_n$ is a small $R$-module.
\end{enumerate}



Suppose, by way of contradiction, that $H^*(\Omega X)$ has exponential growth.
As in Section \ref{sec:zciisgci}, if all the degrees
$|\zeta_0|,...,|\zeta_{n-1}|$ are non-zero,  then we have a contradiction,
since by Lemma \ref{lem:growthintriangles}
$H^*(M_n \otimes_R \Q)$ must have exponential growth because
$H^*(M_0\otimes_R \Q)\cong H^*(\Omega X)$ has exponential growth.

We are left with showing that if $|\zeta_i|=0$ for some $i$, the exponential
growth still propagates. So suppose that $|\zeta_i|=0$ for some $i$ and
$M_i\otimes_R \Q$ has exponential growth. Consider the homotopy colimit
$M_i^{\infty}$ of the telescope $M_i \xrightarrow{\zeta_i}
M_i \xrightarrow{\zeta_i} \cdots$. Since this homotopy colimit is part of a
distinguished triangle: $\oplus_{n=0}^\infty M_i \xrightarrow{1-\zeta_i}
\oplus_{n=0}^\infty M_i \to M_i^{\infty}$ where the first map is in $\Der_i$, $M_i^{\infty}$ is also in $\Der_i$.

Next we show that $M_i^{\infty}$ is finitely built from $\Q$.
By construction $H^*(M_i)$ is non-zero only in finitely many degrees.
Each cohomology group $H^j(M_i)$ is a finite dimensional vector space on
which $\zeta_i$ acts as a linear transformation. For  large enough $m$,
the kernel of $\zeta_i^m$ stabilizes. Denote this kernel by $K \subseteq
H^j(M_i)$, so that
we can write
$$H^j (M_i) \cong K \oplus V, $$
where $\zeta_i^m$ is zero on $K$ and is an isomorphism on $V$.
We see that $H^j (M_i^{\infty})\cong V$ is also finite dimensional, and
$M_i^{\infty}$ itself is finitely built from $\Q$.

The morphism $M_i^{\infty} \xrightarrow{\zeta_i} M_i^{\infty}$
 is an equivalence, so Condition (4) of the definition of nci
shows $M_i^{\infty}\simeq0$.
As $M_i^\infty \otimes_R \Q\simeq 0$ and $H^j(M_i\tensor_R \Q)$ is finite dimensional for each $j$, it follows that
$H^j(\zeta_i \tensor_R \Q)$ is nilpotent on $H^j(M_i\tensor_R \Q)$ for each $j$, i.e., for each $j$ there exists an $n$
such that $\zeta_i^n \tensor_R \Q$ induces the zero map on $H^j(M_i\tensor_R \Q)$.

First observe  that $[M_i,M_i]_R$ (i.e., the ring of
degree 0 homotopy endomorphisms) is a finite dimensional algebra. Indeed,
we start by observing that $[M_0,M_0]_R^*=[\Q,\Q]_R^*$ is finite
dimensional in each degree, and deduce the same for $[M_1,M_1]_R^*,
[M_2,M_2]_R^*, \ldots , [M_i,M_i]_R^*$ using the defining triangles.

Now let $z$ be the morphism
$M_i\tensor_R \Q \xrightarrow{\zeta_i\tensor_R \Q} M_i\tensor_R \Q$, and
note that  the span of $\{z,z^2,z^3...\}$ inside
$[M_i\otimes_R \Q,M_i \otimes_R \Q]_{C_*(\Omega X)}\cong [M_i,M_i]_R$ is finite
dimensional. We now resort to a classical trick to show that $z$ is nilpotent
on $H^*(M_i\otimes_R \Q)$, i.e., that there is some $N$ for which $z^N$
induces the zero map on $H^*(M_i\otimes_R \Q)$.

Suppose $\{z,z^2,...,z^n\}$ is a basis for the span of $\{z,z^2,z^3, ...\}$.
Let $x \in H^j(M_i\otimes_R \Q)$ for some $j$ and let $\{z(x),z^2(x), ...,z^k(x)\}$
be a basis for the span of $\{z(x),z^2(x), z^3(x), ...\}$. Clearly $k\leq n$.
Suppose that $z^{m}(x)\neq 0$ and $z^{m+1}(x)=0$, so that $m\geq k$.
We will show that $m=k$ so $z^{k+1}(x)=0$.

 Write
\[z^{m}(x)=a_1z(x)+a_2z^2(x)+\cdots +a_jz^j(x),\]
where $a_j \neq 0$ and $j\leq k$. First, we note that $j=k$. Indeed, applying
$z$ we find  $0=z^{m+1}(x)=\sum_{i=1}^j a_i z^{i+1}(x)$, so that
 if $j<k$ we get a linear dependence. Thus, for some $t\leq k$
we have
\[ z^{m}(x)=a_t z^t(x)+a_{t+1} z^{t+1}(x) + \cdots +a_k z^k(x),\]
where $a_t \neq 0$. It suffices to show $t=k$, so we suppose $t<k$
and deduce a contradiction. If $t<k$ we apply $z$ to our equation and
deduce
\[ z^{m+1}(x) = 0 = a_t z^{t+1}(x)+a_{t+1} z^{t+2}(x) + \cdots +a_{n-1} z^k(x) + a_k z^{k+1}(x)\]
which means that $z^{k+1}(x)$ is in the span of $\{z^{t+1}(x),...,z^{k}(x)\}$.
Applying $z$ repeatedly, we deduce $z^{k+s}$ is also in the span of
$\{z^{t+1}(x),...,z^{k}(x)\}$ for all $s\geq 1$. Now either $m=k$ and we
are done, or $m=k+s$ for some $s\geq 1$ and we obtain a contradiction.
Accordingly,  $t=k$ and $z^{k+1}(x)=0$ as required.

Since $H^* (M_i \otimes_R k)$ has exponential growth, it follows
that the kernel of the
map $H^*(z): H^* (M_i \otimes_R k) \to
H^* (M_i \otimes_R k)$ has exponential growth.
In particular, this implies that the cone of $z=\zeta_i \otimes \Q: M_i \otimes_R \Q \to M_i \otimes_R \Q$ has exponential growth.
\end{proof}

\subsection{The  first unravelling move.}
We now describe three constructions we may use to build a new nci space
$X$ from a given nci space $X'$. In practice we are given  $X$,
and we unravel the process to obtain $X'$ in such a way that if $X'$
is nci,  so too is $X$. Only the last of these three was necessary in the
 {eci }
case (it was the critical role of Proposition \ref{prop:NoethHurepi} to
show this). We work entirely algebraically, so that $R$ is a model for $X$
and $R'$ is a model for $X'$.

The first move is eliminating an even generator that is also a cocycle.
\begin{lemma}
\label{lem: Unraveling by even element}
Let $R=(\Lambda V,d)$ be a minimal Sullivan algebra and let $x\in V$ be an element of even degree such that $dx=0$. Then multiplication by $x$ yields a natural transformation on $\Der (R)$: $M \xrightarrow{x\cdot} \Sigma^{-|x|}M$ whose cone is $R/(x) \otimes_R M$. If $R/(x)$ is nci of length $\leq n$ then $R$ is nci of length $\leq n+1$.
\end{lemma}

We record the topological counterpart of this lemma.
\begin{lemma}
Let $x\in \pi_{2n}^{\vee}(X)$ be an element that is in the image of the dual Hurewicz map $H^*(X) \to \pi_*^{\vee}(X)$. Then there is a fibration sequence:
\[ X' \to X \to K(\Q x)\]
such that $x$ is in the image of $\pi_{2n}^\vee(K(\Q x))$. If $X'$ is nci of length $\leq n$ then $X$ is is nci of length $\leq n+1$.
\end{lemma}

\begin{proof}[of Lemma~\ref{lem: Unraveling by even element}]
In this case there is a short exact sequence of dg-$R$-modules:
\[ \Sigma^{|x|} R \hookrightarrow R \twoheadrightarrow R/(x) .\]
The leftmost map is given by multiplication by $x$: $a \mapsto a\cdot x$.
We write $V=W \oplus \Q x$ so that $R/(x)\cong (\Lambda W,d)$, isomorphic
to a sub-DGA of $R$.

Define a natural transformation $\zeta: 1_{\Der (R)} \to \Sigma^{-|x|} 1_{\Der (R)}$ as multiplication by $x$. If $M$ is a cofibrant dg-$R$-module, then applying $-\otimes_R M$ to the short exact sequence above yields a distinguished
triangle
\[ \Sigma^{|x|} M \xrightarrow{\zeta} M \to R/(x) \otimes_RM\]
in $\Der (R)$.

Now suppose that $R/(x)$ is nci of length $\leq n$, so that
 there are subcategories $\Der_{R/(x)}=\Der_0 \supseteq \Der_1 \supseteq \cdots \supseteq \Der_n$ and appropriate natural transformations $\zeta_0,...,\zeta_{n-1}$. The map $p:R \to R/(x)$ of CDGAs induces an obvious functor
$p^*:\Der (R/(x)) \to \Der (R)$. Define $\Der'_i=p^* \Der_i$ and similarly $\zeta'_i=p^*\zeta_i$. Set $\Der'_{-1}=\Der (R)$ and let $\zeta'_{-1}$ be the natural transformation $\zeta$ defined above.

We may check that the subcategories $\Der'_{-1},...,\Der'_{n-1}$ and natural
transformations $\zeta'_{-1},...,\zeta'_{n-1}$ satisfy the conditions for being
nci. One need only note three things. First,
if $M$ is a small dg-$R/(x)$-module then $p^*M$ is a small dg-$R$-module,
because $R/(x)$ is a small $R$-module. Second, if $N$ is a dg-$R$-module
finitely built by $\Q$, then so is $R/(x)\otimes_R N$. The reason is that
$N$ is finitely built by $\Q$ over $R$ if and only if $H^*(N)$ is finite
dimensional. Third, if $N \neq 0$ is finitely built by $\Q$, then $N/\zeta$
is not zero,  because $|\zeta|\neq 0$ and so $\zeta$ cannot induce an
isomorphism of $H^*(N)$.
\end{proof}

\subsection{The  second unravelling move.}
The second move eliminates an even cocycle by adding an odd generator.
The argument is essentially the same as the previous one
(except that the cocycle is not a generator),  so we omit the proof.
\begin{lemma}
\label{lem: Unraveling an even cocycle that is not a generator}
Let $R=(\Lambda V,d)$ be a minimal Sullivan algebra and let $f\in R$ be an even cocycle. Let $R'=(\Lambda (V\oplus \Q y),d')$,
 where $d'y=f$ and $d'v=dv$ for all $v\in V$. Then multiplication by $f$ yields a natural transformation on $\Der (R)$: $M \xrightarrow{f\cdot} \Sigma^{-|f|}M$ whose cone is $R' \otimes_R M$. If $R'$ is nci of length $\leq n$ then $R$ is is nci of length $\leq n+1$.
\end{lemma}

The topological counterpart of this lemma is again a fibration:
\[ X' \to X \to K(\Q f),\]
only this time we just require that $X \to K(\Q f)$ represent a nontrivial element in $H^*(X)$.

\subsection{The third unravelling move.}
Finally, the third move is passing to a subalgebra.
This is the precise counterpart of the argument of Subsection \ref{subsec:gciisscistrategy}, and the conclusion is analogous to a spherical fibration
$$S^{|x|-1}\lra X\lra X'.$$
Nevertheless, we describe how this move fits within the nci context.

\begin{lemma}
\label{lem: Unraveling by odd element}
Let $R=(\Lambda V,d)$ be a minimal Sullivan algebra. Let $x\in V$ be an element of odd degree such that $dv \notin x\Lambda V$ for all $v\in V$. Then there is a sub Sullivan algebra $Q \subset R$ and a natural transformation $\zeta:1_{\Der (R)} \to \Sigma^{|x|+1}1_{\Der (R)}$ such that for any $M\in \Der (R)$ there is a distinguished triangle in $\Der (R)$:
\[M \xrightarrow{\zeta} \Sigma^{|x|+1} M \to \Sigma R\otimes_Q M \ .\]
If $Q$ is nci of length $\leq n$ then $R$ is nci of length $\leq n+1$.
\end{lemma}
\begin{proof}
Write $V=W\oplus \Q x$,  and let $Q$ be the Sullivan algebra $(\Lambda W,d)$.
Clearly $Q$ is a minimal Sullivan algebra and there is an obvious inclusion
$\iota:Q \hookrightarrow R$ of Sullivan algebras. We have seen earlier that
there is a natural transformation $\zeta$ on $\Der (R)$ such that
$M \xrightarrow{\zeta} \Sigma^{|x|+1} M \to \Sigma R\otimes_Q M$ is a distinguished triangle for all $M\in \Der (R)$ .

Suppose that $Q$ is nci of length $\leq n$, so there are subcategories
$\Der_{Q}=\Der_0 \supseteq \Der_1 \supseteq \cdots \supseteq
\Der_n$ and appropriate natural transformations $\zeta_0,...,\zeta_{n-1}$.
The map $\iota:Q \to R$ of CDGAs induces a functor $\iota_*:\Der (Q) \to
\Der (R)$, where $\iota_*(N)$ is the induced dg-$R$-module
$R\otimes_Q N$. Define $\Der'_i=\iota_* \Der_i$ and similarly
$\zeta'_i=\iota_*\zeta_i$. Set $\Der'_{-1}=\Der (R)$ and let
$\zeta'_{-1}$ be the natural transformation $\zeta$ on $\Der (R)$ defined
above.

We may check that the subcategories $\Der'_{-1},...,\Der'_{n-1}$ and natural transformations $\zeta'_{-1},...,\zeta'_{n-1}$ satisfy the conditions for being nci.
One need only note three things. First, if $N$ is a small
DG-$Q$-module then $\iota_*N$ is a small DG-$R$-module. Second,
if $M$ is a dg-$R$-module finitely built by $\Q$, then $M$ is finitely built
by $\Q$ also as a dg-$Q$-module. Third, $|\zeta|\neq 0$ and therefore
$M/\zeta \neq 0$ for every non-zero dg-$R$-module $M$ that is finitely built
by $\Q$.
\end{proof}

\subsection{Two examples.}
We discuss some examples of minimal Sullivan algebras which are nci but
not  {eci}.

\begin{example}
Consider the minimal Sullivan algebra
\[ R=(\Lambda(x_3,y_3,z_3,a_8),dx=dy=dz=0, da=xyz). \]
First, we see that it is nci by showing explicitly how to unravel it.
Indeed, we may apply Lemma~\ref{lem: Unraveling an even cocycle that is not a generator} to the cocycle $xy$ to yield
\[ R'=(\Lambda(x,y,z,w,a),da=xyz, dw=xy).\]
Now, $d(wz)=da$, so by a change of variables $a'=a-wz$ we see that
\[R' \cong (\Lambda(x,y,z,w,a'), da'=0, dw=xy)\]
Now  $R'$ is  {eci } and from its homotopy we see it is of codimension $4$.
It is therefore also nci of length 4, and therefore
 $R$ is nci of length $\leq 5$.

On the other hand, it is not hard to see that $R$ is not  {eci}. Most explicitly,
one may identify the cohomology ring explicitly and observe that it is not
Noetherian: it has a basis of monomials $x^iy^jz^ka^l$ where $i,j,k\in \{0,1\}$
omitting the monomials $xyza^l$ for $l\geq 0$ and $a^l$ for $l\geq 1$.
All elements (except those in codegree zero) are nilpotent.

Note also that the dual Hurewicz map is not surjective in codegree 8, so that
the method of Subsection \ref{subsec:gciisscistrategy} cannot be applied.
\end{example}

\begin{remark}
\label{rem: unextendable natural transformation}
Finally, we can see explicitly why the natural transformation in the nci definition cannot always be extended as we would require for the zci definition.
In the previous example, multiplication by the cocycle $a'=a-wz$ defines a natural transformation on the $\Der (R')$ and therefore also on the image of
$\Der (R')$ under restriction. This natural transformation cannot be
extended to a central natural transformation of $1_{\Der (R)}$, since
we would then have a commutative diagram
\[ \xymatrixcompile{ {\Sigma^{|xy|+|a'|} R} \ar[r]^{xy} \ar[d]^{a'} & {\Sigma^{|a'|}R} \ar[r] \ar[d]^{a'} & {\Sigma^{|a'|}R'} \ar[d]^{a'} \\
{\Sigma^{|xy|} R} \ar[r]^{xy} \ar[d] & R \ar[r] \ar[d] & R \ar[d]^{a'} \\
{\Sigma^{|xy|} B} \ar[r]^{xy} & B \ar[r] & {B'} } \]
of distinguished triangles in $\Der (R)$.
The natural transformation $a'$ must be the zero map on $R$, hence $B\cong R\oplus \Sigma^{|a|}R$. This shows that $B'$ is isomorphic to $R' \oplus \Sigma^{|a|}R'$, and therefore $a'$ acts as a regular element. However this contradicts
the fact that $(a')^2=0$ from the long exact sequence of the triangle.
\end{remark}

\begin{example}
Consider the minimal Sullivan algebra
$$R=(\Lambda(x_5,y_3,z_3,y'_3,z'_3,a_{10}),dx=yz+y'z', da=xyy').$$
First, we see that it is nci by showing explicitly how to unravel it.
We apply Lemma~\ref{lem: Unraveling an even cocycle that is not a generator} to the cocycle $yy'$, yielding:
\[ R'=(\Lambda(x_5,y_3,z_3,y'_3,z'_3,a_{10},w_5),dx=yz+y'z', da=xyy', dw=yy')\]
This yields $d(a+xw)=(dx)w$. We apply Lemma~\ref{lem: Unraveling an even cocycle that is not a generator} twice more, for the cocycles $wy$ and $wy'$, yielding:
\[ R''=(\Lambda(x_5,y_3,z_3,y'_3,z'_3,a_{10},w_5,t_7,t'_7),dx=yz+y'z', da=xyy', dw=yy', dt=wy, dt'=wy')\]
Finally we have $d(a+xw-zt-zt')=0$. So, as in the previous example, we can do a change of variables $a'=a+xw-zt-zt'$ and see that $R''$ is  {eci } of
codimension 8, and hence nci of length $8$. It follows that
 $R$ is $nci$ of length $\leq 11$.

On the other hand, it is not hard to see that $R$ is not  {eci}. Indeed,
since the differential is non-zero on the top even class $a_{10}$,
the dual Hurewicz map is not surjective in codegree 10, so that
the method of Subsection \ref{subsec:gciisscistrategy} cannot be applied.
By Proposition \ref{prop:NoethHurepi}, $H^*(R)$ is not Noetherian and so
$R$ is not  {eci}.
\end{example}

\subsection{Proof of Theorem~\ref{the: Main theorem}}
\label{sub: Proving the main theorem}

Let $R=(\Lambda V,d)$ be a simply-connected minimal Sullivan algebra where
$V$ is finite dimensional. We will show that $R$ is nci. The proof proceeds
by induction on the dimension of $V^\mathrm{even}$.
The induction starts since if $V^\mathrm{even}=0$,  successive applications
of the third unravelling move (Lemma~\ref{lem: Unraveling by odd element})
will reduce to a Sullivan algebra with trivial differential. This is then
the model of a product of odd spheres, which is obviously nci.

If $V^\mathrm{even} \neq 0$, then we apply
Lemma~\ref{lem: eliminating a minimal even generator} below, which
says we may repeatedly apply the second unravelling move (i.e.,
add a finite number of odd generators) until we reach a CDGA $R'$ with an even
generator $a \in V^\mathrm{even}$ such that $da=0$. Now use the first
unravelling move (Lemma~\ref{lem: Unraveling by even element}) on $a$.
The minimal Sullivan algebra $R'/(a)$ is nci by the inductive hypothesis,
so that $R'$ is nci by Lemma~\ref{lem: Unraveling by even element},  and
$R$ is nci by Lemma
\ref{lem: Unraveling an even cocycle that is not a generator}.
\qqed

The key ingredient 
is the following technical result. Note that for a minimal Sullivan algebra
$R=(\Lambda V,d)$ with $V$ of finite type, the image of an element
$[f]\in H^n(R)$ under the dual Hurewicz $h^\vee$ map is non-zero if and only
if there is an isomorphism of minimal Sullivan algebras
$\rho:R \xrightarrow{\cong} (\Lambda V',d')$ such that $\rho(f)\in V'$.
\begin{lemma}
\label{lem: technical lemma for nci}
Let $R=(\Lambda V,d)$ be a minimal (simply connected) Sullivan algebra such
that $V$ is finite dimensional and concentrated only in odd degrees. Let
$0\neq [f]\in H^{2i+1}(R)$ be an odd element of the cohomology of $R$.
If $h^\vee([f])=0$, then there is an element $0\neq [g]\in H^{2j}(R)$ such that $0<j\leq i$.
\end{lemma}
\begin{proof}
Suppose the image of $[f]$ under the dual Hurewicz map is zero. The proof goes by induction on the dimension of $V$. Let $x \in V$ be an element of minimal codegree, therefore $dx=0$ and $h^\vee(x)\neq 0$. Eliminate $x$ by adding an even generator $S=(\Lambda(V\oplus \Q a), da=x)$ ($d$ being defined on $V$ as before). There is a distinguished triangle in $\Der (R)$:
\[ \Sigma^{|x|} S \to R \xrightarrow{\varphi} S \xrightarrow{\psi} \Sigma^{|x|+1} S.\]

Since $S$ is equivalent to the minimal Sullivan
algebra $R/(x)$, the induction assumption holds for $S$.
There are two possible cases.

In the first case $H_*\varphi[f] \neq 0$. The image of $\varphi[f]$ under the dual Hurewicz map is zero
because $\pi_*^\vee R \to \pi_*^\vee S$ is an epimorphism with kernel $\Q x$.
By the induction assumption there is a even degree element in
the cohomology of $S$ whose codegree is smaller than $|f|$. Let $[g]$ be such
an element of minimal degree. If $[g]$ is in the image of $\varphi$, then we
are done. If not, then $\psi[g] \neq 0$. But $\psi[g] \in H^{|g|-|x|+1}(S)$,
which contradicts the minimality of $|g|$ (note that $|g|-|x|+1 >0$,  since
otherwise $g$ is one degree below the minimal generators, which is impossible).

The remaining option is that $\varphi[f]=0$. Hence $f=dw$ for some $w \in S$. Write $w$ as
\[ w=a^n A_n+ a^{n-1} A_{n-1} + \cdots + aA_1+A_0, \]
where $A_i \in \Lambda V$. Clearly $f$ is homologous to $f-dA_0$, so without loss of generality we can assume that $A_0=0$. Also note that all the $A_i$ are of even codegrees smaller than the codegree of $f$ (in fact $|A_i|\leq |f|-|x|$). Calculating $dw$ gives
\[ dw=a^n d(A_n)+\sum_{i=1}^n a^{i-1}(dA_{i-1}+ixA_i).\]
Since $f\in \Lambda V$ we see that:
\begin{align*}
f&=x A_1 \\
dA_{i-1}&=-ixA_i \quad \mathrm{ for } \ i\geq 2\\
dA_n &=0
\end{align*}

Thus $A_n$ is a cocycle. If $A_n$ is not a coboundary in $R$, then we are done. Otherwise there is a $B_n$ so that
\[ A_n=dB_n\]
Now $dA_{n-1}=-nxA_n=-nxdB_n = d(nxB_n)$, whence $A_{n-1}-nxB_n$ is an even codegree cocycle. Again, if it is not a coboundary then we are done. Otherwise
there is a $B_{n-1}$ so that
\[A_{n-1}-nxB_n = dB_{n-1}.\]
Now
\begin{align*}
dA_{n-2}&=-(n-1)xA_{n-1}=-(n-1)x(nxB_n+dB_{n-1}) \\
&= -(n-1)xdB_{n-1} = (n-1)d(xB_{n-1}).
\end{align*}
So we see that $A_{n-2}-(n-1)xB_{n-1}$ is an even codegree cocycle. We continue in this manner until either we get the desired even codegree element in the cohomology of $R$, or we end with
\[A_1-2xB_2=dB_1.\]
But now $f=xA_1=x(2xB_2+dB_1)=xdB_1=-d(xB_1)$, i.e. $[f]=0$, which is a contradiction.
\end{proof}

Using the previous lemma we can now show that the second unravelling move may be
used to make an even generator become a cycle.
\begin{lemma}
\label{lem: eliminating a minimal even generator}
Let $R=(\Lambda V,d)$ be a minimal (simply connected) Sullivan algebra such that $V$ is finite dimensional. Then there is a sequence
$R=R_0 \to R_1 \to \cdots \to R_n$
of unravelling moves of the second type (i.e., moves to which Lemma~\ref{lem: Unraveling an even cocycle that is not a generator} applies)
such that $R_n$ is isomorphic to a minimal
Sullivan algebra: $R'_n=(\Lambda V',\delta)$, where there is a minimal even element
$a\in V'$ such that $\delta a=0$.
\end{lemma}
\begin{proof}
If there is a minimal even codegree element $a \in V$ such that $da=0$ we are done. Otherwise, choose some minimal even codegree element $a \in V$. Let $V_s \subset V$ be the subspace of codegrees smaller than $||a||$. Minimality of the Sullivan algebra implies that $da \in \Lambda V_s$.

Apply the second unravelling move (Lemma~\ref{lem: Unraveling an even cocycle that is not a generator}) to the even elements of $H^*(R)$ starting from the bottom and going up. Hence, let $g \in R$ be a minimal even degree cocycle that is not a coboundary. Applying Lemma~\ref{lem: Unraveling an even cocycle that is not a generator} to $g$ yields
\[ R_1=(\Lambda (V\oplus \Q x_g), dx_g=g)\]
and a map $R \to R_1$. We continue adding odd generators to $R$ in this way until there is no more homology in even codegrees less than or equal to $||a||$. Note that we need only kill cocycles composed of odd generators only, even in codegree $a$.

We end with a minimal Sullivan algebra $R_n=(\Lambda U,d)$, where $R_n$ has no even degree cohomology in dimension $||a||$ or lower.
Let $U_s \subset U$ be the subspace of codegrees smaller than $||a||$. Clearly $U_s$ has only odd degree elements. Define $T=(\Lambda U_s,d)$ to be the appropriate sub-CDGA of $R_n$. Then $T$ also has no cohomology in even codegrees $||a||$ or lower. Because $da$ is a cocyle in $T$, $h^\vee(da)=0$ and $H^{2j}(T)=0$ for $2j\leq ||a||$, it follows from Lemma~\ref{lem: technical lemma for nci} that $da$ is a coboundary in $T$. Thus $da=du$ for some $u \in \Lambda U_s$.

Now $a-u$ is a cocycle in $R_n$. So a change of variables: $a'=a-u$ yields a new minimal Sullivan algebra $R'_n=(\Lambda U',\delta)$ where a minimal codegree even generator $a'$ is a cocycle.
\end{proof}

\appendix

\section{Gorenstein rings and spaces}
\label{sec:Gorenstein}
This appendix discusses h-Gorenstein spaces, emphasizing the duality
this gives. The material comes from \cite{FHTGor}, \cite{DGI1} and
\cite{GL}, but the results have not been brought together
explicitly before.

\subsection{Contents}
Recall that a commutative local Noetherian ring $(R, \mathfrak{m},k)$ is {\em Cohen-Macaulay}
 if its depth is equal to its dimension, and that it is {\em Gorenstein} if
 $R$ is of finite injective dimension as a module. Furthermore, if $R$
is Gorenstein of dimension $r$
$$\Ext_R^*(k,R)= \Ext_R^r(k,R)=k. $$
Despite the definition, the important content of the Gorenstein condition
is a duality property (this will be a special case of one in the  CDGA
case below).

F\'elix-Halperin-Thomas \cite{FHTGor} have considered
the analogue for spaces (which we call the h-Gorenstein condition) at length.
We recall the definition below.  It transpires that for spaces
with finite dimensional cohomology (or finite category)
$X$ is h-Gorenstein if and only if $H^*(X)$ is Gorenstein. Our contribution
is to make explicit the duality statements in the positive dimensional
case following \cite{DGI1}. There is a structural duality statement
at the level of derived categories even when  $H^*(X)$ is not Gorenstein.
Thus if $X$ is h-Gorenstein, there are consequences for the
cohomology ring \cite{GL}: the cohomology ring $H^*(X)$ is
always generically Gorenstein, if it is Cohen-Macaulay,
it is automatically Gorenstein
(and hence its Hilbert series satisfies a functional equation), and if it
 has Cohen-Macaulay defect 1, its Hilbert series satisfies a
suitable pair of  functional equations.

\subsection{The definition.}

We recall the definition from \cite{FHTGor} in the language of \cite{DGI1}.
\begin{defn}
We say that a DGA $A$ is {\em h-Gorenstein of shift $a$} if
$\Hom_{A}(\Q,A)\simeq \Sigma^a \Q$. We say that a space $X$ is
{\em h-Gorenstein} if $C^*(X)$ is h-Gorenstein.
\end{defn}

We begin with the remark that the definition is an invariant of
quasi-isomorphism, so that any particular rational model of the
space $X$ may be used.

\subsection{Gorenstein duality.}

The purpose of the Gorenstein condition is to capture a duality
property. This takes some work to extract. Since the argument is in
\cite{DGI1} we will be brief. Since we are now mixing two sorts of
duality, it is essential to emphasize that the suspension $\Sigma^a$
is {\em homological}: it increases degrees by $a$ (i.e., it reduces
codegrees by $a$).

\begin{prop}\cite{DGI1}
If $A$ is h-Gorenstein of shift $a$ and $H^*(A)$ is 1-connected and
Noetherian, then
there is an equivalence
$$\cell_\Q A \simeq \Sigma^a A^{\vee}, $$
and hence a spectral sequence
$$\lc^*(H^*(A))\Rightarrow \Sigma^a H^*(A)^{\vee}.$$
\end{prop}

\begin{proof}
By definition, we have an equivalence
$\Hom_{A}(\Q,A)\simeq \Sigma^a \Q$ of $A$-modules. To proceed we need
to apply Morita theory, so we consider
 the endomorphism ring $\cE =\Hom_A(\Q,\Q)$.
There is a natural right $\cE$-module structure on
$\Hom_A(\Q,M)$ for any $M$, so the Gorenstein condition gives
an $\cE$-action on $\Q$. However, since
 $A$ is 1-connected, there is a unique  $\cE$-module structure
on  $\Q$. Thus the Gorenstein condition  gives  an equivalence
$$\Hom_{A}(\Q,A)\simeq \Hom_{A}(\Q,\Sigma^a A^{\vee})$$
of $\cE$-modules. Now apply $\tensor_{\cE} \Q$. Since $H^*(A)$ is
Noetherian, $\Q$ is proxy-small in the sense of \cite{DGI1},
and we may  use Morita theory to deduce
$$\cell_\Q (A) \simeq \cell_\Q (\Sigma^a A^{\vee}). $$
Since $A^{\vee}$ is $\Q$-cellular, the $\cell_\Q $ on the right may
be omitted.

This proves the first statement. For the second, by Corollary
\ref{cor:GammaisCellk}  the stable Koszul complex $\Gamma A$
provides a model for $\cell_\Q (A)$.
Using the natural filtration, we obtain the spectral sequence.
\end{proof}

We remark that the spectral sequence collapses if $H^*(A)$ is Cohen-Macaulay
to show $\lc^r(H^*(A))\cong \Sigma^{a+r} H^*(A)^{\vee}$ (where $r$ is the Krull
dimension of $H^*(A)$).  Thus $H^*(A)$ is also Gorenstein, and $a$ is the
classical $a$-invariant.

The spectral sequence also collapses if $H^*(A)$ is of Cohen-Macaulay defect 1,
to give an exact sequence
$$
0 \lra \lc^r(H^*(A))\lra
\Sigma^{a+r} H^*(A)^{\vee}\lra \Sigma \lc^{r-1}(H^*(A))\lra 0.$$
This is discussed in more structural terms in \cite[5.4]{GL}.

\subsection{Functional equations.}

It may be helpful to record the functional equations satisfied by
the Hilbert series $p_A(t)$ of an h-Gorenstein algebra $A$ when $H^*(A)$ is
a Noetherian ring of Cohen-Macaulay defect 0 or 1. The equations are
deduced from the existence of a local cohomology theorem in
\cite[Section 6]{GL}. Since $H^*(A)$ is cograded, we take $t$ to
be of codegree 1 (i.e., of degree $-1$).

If $H^*(A)$ is Cohen-Macaulay we have
$$p_A(1/t)=(-t)^rt^a p_A(t).$$

If $H^*(A)$ is of Cohen-Macaulay defect 1, we have a pair of functional
equations (introduced in the group theoretic context by Benson and
Carlson \cite{BC2})
$$p_A(1/t)-(-t)^rt^{a}p_A(t)=(-1)^{r-1}(1+t)\delta_A(t)$$
and
$$\delta_A (1/t)=(-t)^{r-1}t^a\delta_A (t), $$
and in fact $\delta_A (t)$ is the Hilbert series of $\lc^{r-1}(H^*(A))^{\vee}$.

\subsection{Examples.}
First we show that there are many familiar examples of h-Gorenstein
DGAs.

\begin{cor}\cite[3.2(ii)]{FHTGor}
\label{cor:GorishGor}
If $H^*(A)$ is Gorenstein  then $A$ is  h-Gorenstein.
\end{cor}

\begin{proof}
If $H^*(A)$ is Gorenstein then the $E_2$-term of the spectral sequence
$$\Ext_{H^*(A)}^{*,*}(\Q,H^*(A))\Rightarrow H^*(\Hom_A(\Q,A))$$
degenerates to an isomorphism
$$\Ext^r_{H^*(A)}(\Q,H^*(A))=\Sigma^{r+a} \Q,  $$
where $a$ is the conventional $a$-invariant.
The spectral sequence therefore collapses to show $A$ is h-Gorenstein
with shift $a$.
\end{proof}

\begin{cor}\cite[3.6]{FHTGor}
If $H^*(A)$ is finite dimensional then
$A$ is  h-Gorenstein if and only if $H^*(A)$ is a Poincar\'e duality
algebra.
\end{cor}

\begin{proof}
If $H^*(A)$ is a Poincar\'e duality algebra of formal dimension
$n$ then it is a zero dimensional Gorenstein ring with $a$-invariant
$-n$, so $A$ is h-Gorenstein with shift $-n$ by the previous corollary.

Conversely, if $A$ is h-Gorenstein of shift $a$, we have a Gorenstein
duality spectral sequence. Since $H^*(A)$ is finite dimensional,  it
is all torsion. Accordingly, $\lc^*(H^*(A))=H^*(A)$,
and the spectral sequence reads
$$H^*(A)=\Sigma^a H^*(A)^{\vee}$$
and $H^*(A)$ is a Poincar\'e duality algebra of formal dimension $-a$.
\end{proof}

One may use these to construct other examples which are h-Gorenstein
but not Gorenstein.

\begin{prop} \cite[4.3]{FHTGor}
Suppose  we have a fibration $F\lra E \lra B$ with
$F$ finite. If $F$ and $B$ are h-Gorenstein with
shifts $f$ and $b$ then $E $ is h-Gorenstein with
shift $e=f+b$.\qqed
\end{prop}

This allows us to construct innumerable examples. For example
any finite Postnikov system is h-Gorenstein \cite[3.4]{FHTGor},
so that in particular any sci space is h-Gorenstein.
A simple example will illustrate the duality.

\begin{example}
\label{eg:hGornotGor}
We construct a rational space $X$ in a fibration
$$S^3\times S^3 \lra X \lra \CP^{\infty} \times \CP^{\infty},  $$
so that $X$ is h-Gorenstein. We will calculate $H^*(X)$ and
observe that it is not Gorenstein.

Let $V$ be a graded vector space with two generators
$u,v$ in degree 2, and let $W$ be a graded vector space with two
generators in degree 4. The two 4-dimensional cohomology classes
 $u^2, uv$ in $H^*(KV)=\Q[u,v]$ define a map $KV\lra KW$,
and we let $X$ be the fibre, so we have a fibration
$$S^3\times S^3 \lra X \lra KV$$
as required.  By \cite{DGI1}, this is h-Gorenstein
with shift $-4$ (being the sum of the shift (viz $-6$) of
$S^3\times S^3$ and the shift (viz 2) of $KV$).

It is amusing to calculate the cohomology ring of
$X$. It is $\Q[u,v,p]/(u^2,uv,up, p^2)$ where $u,v$ and $p$
have degrees $2,2$ and $5$. The dimensions of its graded components
are $1,0,2,0,1,1,1,1,1,\ldots$  (i.e., its Hilbert series
is $p_X(t)=(1+t^5)/(1-t^2)  +t^2$, where $t$ is of codegree 1).

In calculating local cohomology it is useful to note that
${\frak{m}}=\sqrt{(v)}$.
The local cohomology is $\lc^0(H^*(X))=\Sigma_2\Q$ in degree 0
(so that $H^*(X)$ is not Cohen-Macaulay)
and as a $\Q[v]$-module $\lc^1(H^*(X))$ is
$\Q[v]^{\vee}\tensor (\Sigma^{-3} \Q \oplus \Sigma^2 \Q)$.
Since there is no higher local cohomology the local cohomology spectral
sequence necessarily collapses, and the resulting
exact sequence
$$0\lra \lc^1(H^*(X))\lra \Sigma^{-4}H^*(A)^{\vee} \lra \Sigma^{-2}\Q\lra 0$$
is consistent.

Since the Cohen-Macaulay defect here is 1, we have a pair of functional
equations
$$p_X(1/t)-(-t)t^{-4}p_X(t)=(1+t)\delta(t)$$
and
$$\delta (1/t)=t^4\delta (t).$$
Indeed, the first equation  gives $\delta (t)=t^{-2}$, which is indeed
the Hilbert series of $\lc^0(H^*(X))^{\vee}$, and it obviously satisfies
the second equation.
\end{example}


\begin{thebibliography}{999}

\bibitem{Avramov} L.L.Avramov
    ``Modules of finite virtual projective dimension.''
     Invent. Math. 96 (1989), no. 1, 71--101.
\bibitem{AvramovCRM} L.L.Avramov
    ``Infinite free resolutions.''
     Six lectures on commutative algebra (Bellaterra, 1996),
      1--118, Progr. Math., 166, Birkhäuser, Basel, 1998.
\bibitem{AvramovAQ} L.L. Avramov
   ``Locally complete intersection homomorphisms and a conjecture of
    Quillen on the vanishing of cotangent homology.''
    Ann. of Math. (2) 150 (1999), no. 2, 455--487.
\bibitem{AB} L.L.Avramov and R.-O. Buchweitz
     ``Support varieties and cohomology over complete intersections.''
      Invent. Math. 142 (2000), no. 2, 285--318.
\bibitem{AFH} L.L.Avramov, H.-B.Foxby and S.Halperin
     ``Resolutions for DG modules''
       In preparation
\bibitem{AH} L.L.Avramov and S.Halperin
     ``Through the looking glass: a dictionary between rational homotopy theory and local algebra.''
 Algebra, algebraic topology and their interactions (Stockholm, 1983),
  Lecture Notes in Math., {\bf 1183}, Springer-Verlag (1986), 1--27.
\bibitem{BC1} D.J.Benson and J.F.Carlson
   ``Projective resolutions and Poincar\'e duality complexes.''
    Trans. Amer. Math. Soc. 342 (1994), no. 2, 447--488.
\bibitem{BC2} D.J.Benson and J.F.Carlson
  ``Functional equations for Poincar\'e series in group cohomology. ''
    Bull. London Math. Soc. 26 (1994), no. 5, 438--448.
\bibitem{BGzci} D.J.Benson and J.P.C.Greenlees
   ``Complete intersections and derived categories''
    Preprint (2008)
\bibitem{pzci} D.J.Benson, J.P.C.Greenlees and S.Shamir
   ``Complete intersections and mod $p$-cochain algebras''
    Preprint (2008)
\bibitem{DVP}    N. Dupotn and M. Vigu\'e-Poirrier
``Finiteness Conditions for Hochschild Homology
Algebra and Free Loop Space Cohomology Algebra''
 K-Theory {\bf 21} (2000) 293Ð300
\bibitem{tec}
   W.G.Dwyer and  J.P.C.Greenlees
  ``Complete modules and torsion modules.''
   AJM {\bf 124} (2002) 199-220
\bibitem{DGI1} W.G.Dwyer, J.P.C.Greenlees and S.Iyengar
    ``Duality in algebra and topology.''
     Adv. Math. 200 (2006), no. 2, 357--402.
\bibitem{DGI2} W.G.Dwyer, J.P.C.Greenlees and S.Iyengar
     ``Finiteness in derived categories of local rings.''
     Comment. Math. Helv. 81 (2006), no. 2, 383--432.
\bibitem{EM} S.Eilenberg and J.C.Moore ``Homology and fibrations I.
     Coalgebras, cotensor product and its derived functors''
     Comm. Math. Helv. {\bf 40} (1966) 199-236
\bibitem{EKMM} A.D.Elmendorf, I.Kriz, M.A.Mandell and J.P.May
    ``Rings, modules, and algebras in stable homotopy theory.''
      With an appendix by M. Cole. Mathematical Surveys and Monographs, 47.
       American Mathematical Society, Providence, RI, 1997. xii+249 pp.
      ISBN: 0-8218-0638-6
\bibitem{FHTGor} Y.F\'elix, S.Halperin and J.-C.Thomas
    ``Gorenstein spaces.''
    Advances in Maths. {\bf 71} (1988) 92-112
\bibitem{FHT} Y.F\'elix, S.Halperin and J.-C.Thomas
    ''Rational homotopy theory.''
    Graduate Texts in Mathematics, 205.
    Springer-Verlag, New York, 2001. xxxiv+535 pp. ISBN: 0-387-95068-0
\bibitem{GL} J.P.C.Greenlees and G.Lyubeznik
    ``Rings with a local cohomology theorem and applications to cohomology
    rings of groups.''
    J. Pure Appl. Algebra 149 (2000), no. 3, 267--285.
\bibitem{Gulliksen} T.H. Gulliksen
    ``A homological characterization of local complete intersections.''
     Compositio Math. 23 (1971), 251--255.
\bibitem{Gulliksen2} T.H.Gulliksen
     ``A change of ring theorem with applications to Poincar\'e series and
     intersection multiplicity.''
      Math. Scand. 34 (1974), 167--183.
\bibitem{GuLe} T.H.Gulliksen and G.Levin
     ``Homology of local rings.''
      Queen's Paper in Pure and Applied Mathematics, No. 20
     Queen's University, Kingston, Ont. 1969 x+192 pp.
\bibitem{Krause} H.Krause
     ``The stable derived category of a Noetherian scheme.''
      Composition Math. {\bf 141} (2005) 1128-1162
\bibitem{Mandell} M.A.Mandell
   ``$E\sb \infty$ algebras and $p$-adic homotopy theory.''
   Topology 40 (2001), no. 1, 43--94.
\bibitem{Matsumura} H.Matsumura
          ``Commutative ring theory''
           CUP (1986)
\bibitem{MM} J.Milnor and J.C.Moore
       ``On the structure of Hopf algebras.''
         Ann. of Math. (2) 81 1965 211--264.
\bibitem{RS} M. Rothenberg and N.E.Steenrod
         ``The cohomology of classifying spaces of H-spaces.''
          Bull. AMS {\bf 71} (1965) 872-875
\bibitem{Shamir} S.Shamir
       ``Cellularization over group-rings with  applications to the
    Eilenberg-Moore spectral-sequence''
    AGT (to appear) 35pp
\bibitem{Shipley} B.E.Shipley,
     ``$H\Bbb Z$-algebra spectra are differential graded algebras.''
      Amer. J. Math. 129 (2007), no. 2, 351--379.
\end{thebibliography}
\end{document}